\documentclass[a4paper,11pt]{amsart}

\usepackage{graphicx}
\usepackage{amsmath}
\usepackage{amsthm}
\let\oldproofname=\proofname
\renewcommand{\proofname}{\rm\bf{\oldproofname}}
\usepackage{amssymb}
\usepackage[active]{srcltx}
\usepackage{hyperref}
\usepackage{enumerate}
\newtheorem{thm}{Theorem}
\newtheorem{lem}[thm]{Lemma}
\newtheorem{prop}[thm]{Proposition}
\newtheorem{cor}[thm]{Corollary}
\makeatletter
\newcommand{\neutralize}[1]{\expandafter\let\csname c@#1\endcsname\count@}
\makeatother

\newenvironment{thmbis}[1]
  {\renewcommand{\thethm}{\ref*{#1}$'$}%
   \neutralize{thm}\phantomsection
   \begin{thm}}
  {\end{thm}}

\theoremstyle{definition}
\newtheorem{defn}{Definition}[section]

\theoremstyle{remark}
\newtheorem{remark}{Remark} 

\theoremstyle{plain}

\def\CC{{\mathbb C}}

\def\HH{{\mathbb H}}
\def\NN{{\mathbb N}}

\def\RR{{\mathbb R}}

\def\ZZ{{\mathbb Z}}

\def\vecu{{\text{\boldmath$u$}}}
\def\vecv{{\text{\boldmath$v$}}}

\def\vecw{{\text{\boldmath$w$}}}

\def\scrC{{\mathcal C}}

\def\scrF{{\mathcal F}}
\def\scrH{{\mathcal H}}

\def\scrI{{\mathcal I}}

\def\scrM{{\mathcal M}}

\def\scrP{{\mathcal P}}

\def\scrS{{\mathcal S}}

\def\scrU{{\mathcal U}}

\def\scrY{{\mathcal Y}}
\def\scrZ{{\mathcal Z}}

\def\fB{{\mathfrak B}}

\def\Re{\operatorname{Re}}
\def\Im{\operatorname{Im}}

\def\G{\operatorname{G{}}}

\def\S{\operatorname{S{}}}

\def\SL{\operatorname{SL}}

\def\Onder#1#2#3#4#5{#1 \setbox0=\hbox{$#1$}\setbox1=\hbox{$#2$}
       \dimen0=.5\wd0 \dimen1=\dimen0 \dimen2=\dp0 \dimen3=\dimen2
       \advance\dimen0 by .5\wd1 \advance\dimen0 by -#4
       \advance\dimen1 by -.5\wd1 \advance\dimen1 by -#4
       \advance\dimen2 by -#3 \advance\dimen2 by \ht1
       \advance\dimen2 by 0.3ex \advance\dimen3 by #5
        \kern-\dimen0\raisebox{-\dimen2}[0ex][\dimen3]{\box1}
       \kern\dimen1}

\newcommand{\GaG}{\Gamma\backslash G}

\newcommand{\sfrac}[2]{{\textstyle \frac {#1}{#2}}}

\newcommand{\matr}[4]{\left( \begin{matrix} #1 & #2 \\ #3 & #4 \end{matrix} \right) }

\newcommand{\smatr}[4]{\left( \begin{smallmatrix} #1 & #2 \\ #3 & #4 \end{smallmatrix} \right) }

\newcommand{\fg}{\mathfrak{g}}

\begin{document}
\title[Equidistribution of expanding horospheres]{On the rate of equidistribution of expanding horospheres in finite-volume quotients of $\mathrm{SL}(2,\mathbb{C})$.}

\author{Samuel Edwards}

\address{}
\email{}


\date{\today}
\thanks{}

\subjclass[2000]{}


\begin{abstract}
Let $\Gamma$ be a lattice in $G=\mathrm{SL}(2,\mathbb{C})$. We give an effective equidistribution result with precise error terms for expanding translates of pieces of horospherical orbits in $\Gamma\backslash G$. Our method of proof relies on the theory of unitary representations.
\end{abstract}

\maketitle
\section{Introduction}\label{intro}
Let $G$ be a connected Lie group, and $\Gamma$ a lattice in $G$. Various properties of orbits of horospherical subgroups of $G$ in the homogeneous space $\GaG$ have been studied by a multitude of authors. In particular, the effective equidistribution of expanding translates of such orbits has been established, cf. e.g. Kleinbock and Margulis \cite{KleinMarg}. 

For $G=\SL(2,\RR)$, precise results relating the rate of equidistribution of orbits of the horocycle flow with the spectral theory of $\GaG$ have been obtained by a number of authors, e.g. Burger \cite{Burger90}, Flaminio and Fornio \cite{FlamForn}, Hejhal \cite{Hejhal}, Sarnak \cite{Sarnak}, Selberg (unpublished), Str\"ombergsson \cite{Strom, Strom2}, and Zagier \cite{Zagier}. One of the key parts of \cite{Burger90} is an integral formula \cite[Lemma 1 (A)]{Burger90} for certain operators in irreducible unitary representations of $\SL(2,\RR)$. More specifically, the formula relates horospherical averages with the action of a corresponding diagonal subgroup in an arbitrary irreducible unitary representation. One of the results in \cite{Burger90} obtained by use of the formula is an explicit rate of equidistribution for averages along horocycle orbits when $\Gamma$ is \emph{cocompact}. In \cite{Strom}, use of the same formula is extended to prove explicit rates of equidistribution for \emph{non-cocompact} lattices. In an ongoing project, we aim to generalize this method to obtain rates of equidistribution for other Lie groups. This note is a first report on this project, discussing only the case of $\SL(2,\CC)$. This case has the benefits of allowing us to be completely explicit, and permitting comparisons with similar previously known results for hyperbolic 3-orbifolds. As we shall we see, a number of complications arise compared with $\SL(2,\RR)$, both in expressing an integral formula corresponding to that of Burger (which we do in Proposition \ref{Intformlem}), as well as in the application of it to the problem of obtaining explicit rates of equidistribution.

From now on let $G=\SL(2,\CC)$, and $\Gamma$ a lattice in $G$. We denote by $\mu_G$ the unique Haar measure on $G$ such that the pushforward measure $\mu$ (under the map $g\mapsto \Gamma g$) on $\GaG$ is a probability measure. The group $G$ acts on $\GaG$ by right translation, and the action of the following subgroups of $G$ will be of particular interest to us:
\begin{equation*}
A=\left\lbrace a_t=\smatr {e^{t/2}} {0} {0} {e^{-t/2}} \;:t\in\RR\right\rbrace,
\end{equation*}
and
\begin{equation*}
N=\lbrace n_z=\smatr {1} {z} {0} {1} \;:z\in\CC\rbrace.
\end{equation*}
Note that $N$ is a horospherical subgroup of $G$, relative to $A$, i.e.
\begin{equation*}
N=\left\lbrace g\in G\,:\, \lim_{t\rightarrow \infty} a_{-t} g a_{t} = \smatr 1 0 0 1 \right\rbrace.
\end{equation*}
Let $\mu_N$ denote the Haar measure on $N$, chosen so that $\mu_N$ is the pushforward measure of the Lebesgue measure on $\CC$ under the map $z\mapsto n_z$. The maximal compact subgroup $\mathrm{SU}(2)$ of $G$ is denoted by $K$. The symmetric space $G/K$ can be identified with the three-dimensional hyperbolic upper half-space, and $\scrM:=\GaG/K$ is a finite-volume hyperbolic 3-orbifold ($\scrM$ is a manifold if $\Gamma$ is torsion-free). We use $\Delta$ to denote the Laplace-Beltrami operator on $\scrM$. Let $\lambda_1$ be the smallest positive eigenvalue for $-\Delta$ acting on $L^2(\scrM)$, and define $s_1\in [1,2)$ by
\begin{equation*}
s_1=\begin{cases} 1+\sqrt{1-\lambda_1}&\quad \mathrm{if}\,\,\lambda_1\in (0,1)\\ 1&\quad \mathrm{otherwise}.\end{cases}
\end{equation*}
In order to state our main result, we must introduce a Sobolev norm $\|\cdot\|_{W^m}$ on functions in $L^2(\GaG)$; it is discussed in greater detail in Section \ref{L2decomp}. We also let $\scrY_{\Gamma}$ denote the \emph{invariant height function} on $\GaG$. A stringent definition of $\scrY_{\Gamma}$ is given in Section \ref{HgtSec} for non-cocompact $\Gamma$, and in Section \ref{Sobbds} for cocompact $\Gamma$. For now it suffices to view this function as a measure of how far into a cusp of $\GaG$ a point $p$ lies; for a fixed $x_0\in \mathcal{M}$, let $\mathrm{dist}(p)$ denote the distance between $x_0$ and the image of $p$ in $\scrM$. Then $\scrY_{\Gamma}(p)$ is comparable to $e^{\mathrm{dist}(p)}$.
\begin{thm}\label{mainthm}
Let $B'$ be a connected compact subset of positive Lebesgue measure in $\CC$ and assume that there exists a piecewise smooth, bi-Lipschitz mapping of the circle to $\partial B'$. If $B=\lbrace n_z\in N\,:\; z\in B'\rbrace$, then there exists a constant $C(\Gamma,B')>0$ such that for all $T\geq 0$, all $p\in \GaG$, and all $f\in L^2(\GaG)$ with $\|f\|_{W^7}<\infty$, 
 \begin{align*}
\left| \!\frac{1}{\mu_N(B)}\!\!\int_{B}\!\! f\big(pna_{-T}\big)\,d\mu_N(n)\!-\!\int_{\GaG}\!f\,d\mu\right|\!\leq C(\Gamma,B')\|f\|_{W^7} \Big\lbrace  \big(e^{-T}\scrY_{\Gamma}(p)&\big)^{2-s_1}+e^{-T}T^4\\&+e^{-T}(1+T^3)\scrY_{\Gamma}(p)\! \Big\rbrace.
\end{align*}
\end{thm}
We also prove a strengthened version of the above, Theorem \ref{mainthmprime} (cf. p. \pageref{mainthmprime}), where a bound on $C(\Gamma,B')$'s dependency on the set $B'$ is given. 

By using the rate of exponential mixing (cf. \cite{Shalom}, \cite[Proposition 5.3]{Kont}) and the so-called ``Margulis' mixing trick", one may get a similar result (cf. \cite[Proposition 2.4.8]{KleinMarg}). The rate one obtains in this manner, however, is worse than that which is achieved in Theorem \ref{mainthm}. This is discussed in greater detail in the introduction to \cite{Sod}. In \cite{Sod}, S\"odergren proves related results regarding the rate of effective equidistribution of pieces of closed horospheres in hyperbolic $n$-manifolds. Our Theorem \ref{mainthm} thus generalizes certain special cases of these results when $n=3$, proving that this type of equidistribution even holds for translates of pieces of non-closed horospheres, and for test functions on the frame bundle of the manifold. We give an explicit statement of one such generalisation in Corollary \ref{mainthmprimecor} (cf. pg. \pageref{mainthmprimecor}).

Since the bounds in Theorem \ref{mainthm} are uniform in $p$, we may study the equidistribution of horospherical orbits by considering the point $pa_T$, giving
\begin{cor}\label{mainthmcor}
Let $B$ satisfy the conditions of Theorem \ref{mainthm}, and define $B_T:=a_T B a_{-T}$. Then for all $T\geq 0$, all $p\in \GaG$, and all $f\in L^2(\GaG)$ such that $\|f\|_{W^7}<\infty$,
 \begin{align*}
\left| \!\frac{1}{\mu_N(B_T)}\!\!\int_{B_T}\!\! f\big(pn\big)\,d\mu_N(n)\!-\!\int_{\GaG}\!f\,d\mu\right|\!\leq C(\Gamma,B')\|f\|_{W^7}\Big\lbrace  \big(e^{-T}\scrY_{\Gamma}&(pa_T)\big)^{2-s_1}+e^{-T}T^4\\&\;\;\;+e^{-T}(1+T^3)\scrY_{\Gamma}(pa_T)\! \Big\rbrace.
\end{align*}
\end{cor}
The corresponding result for $\SL(2,\RR)$ is proved and stated as Proposition 3.1 in \cite{Strom}, the proof of which is the inspiration for the proof of Theorem \ref{mainthm}. The equidistribution properties of orbits of horospherical subgroups are well-known; in fact, the celebrated results of Ratner give a precise understanding of the asymptotic behaviour of orbits of arbitrary \emph{unipotent} subgroups. 

It is well-known that for a given point $p$, the horospherical orbit $pN$ is either \emph{closed} or \emph{dense}. Furthermore,  $pN$ is dense precisely when $pa_T$ is recurrent; there are therefore no closed horospheres in $\GaG$ when $\Gamma$ is cocompact. Noting that for cocompact $\Gamma$, $\scrY_{\Gamma}(p)\ll 1$,  Corollary \ref{mainthmcor} then provides an explicit rate of equidistribution for averaging sequences $B_T$ for \emph{every} horosphere in $\GaG$. When $\Gamma$ is \emph{non-cocompact}, however, there are many $p$ for which $pN$ is a closed horosphere. Moreover, there exist points $p$ for which $pa_T$ is recurrent, but $\lim_{T\rightarrow \infty} e^{-T}T^3\scrY_{\Gamma}(pa_T)\neq 0$ (thus Corollary \ref{mainthmcor} does not by itself give effective equidistribution of every non-closed horosphere). Letting $E$ be the exceptional set of such $p$, and $Q$ be the parabolic subgroup of upper triangular matrices in $G$, we observe that $EQ=E$; whether a point $p$ is in $E$ or not is therefore completely determined by $p$'s image in $\GaG/Q$. We also note that, in a certain sense, $E$ is small; by \cite[Theorem 1]{Melian}, the image of $E$ in $G/Q$ has Hausdorff dimension zero. Note that the quotient $G/Q$ may be identified with the ``boundary sphere" of hyperbolic 3-space.

In \cite{Strom}, effective equidistribution of \emph{every} non-closed horocycle is achieved (\cite[Theorem 1]{Strom}) by carefully splitting the horocycle into a number of pieces and using \cite[Proposition 3.1]{Strom} on all but one exceptional piece. Str\"ombergsson imposes an additional weighted supremum norm on the functions that are considered; it is this norm which is used to control the contribution from the exceptional piece of the horocycle. Moreover, it is shown in \cite[Proposition 4.1]{Strom} that one may not replace this supremum norm by a Sobolev norm \emph{of any order}. We believe that by a similar type of argument, one should be able to obtain an effective equidistribution result for \emph{all} non-closed horospheres in $\GaG$. We do not attempt this here, however. 

\subsection{Acknowledgements}
The research leading to these results was funded by Swedish Research Council Grant 621-2011-3629 and by a grant from the G\"oran Gustafsson Foundation for Research in Natural Sciences and Medicine. I would like to thank my advisor Andreas Str\"ombergsson for his time and effort in assisting me with this project; it is greatly appreciated. 
\section{Preliminaries}\label{prelims}
We recall some facts regarding the structure of $G$, its Lie algebra, invariant measures and representations. The main references for this section are \cite[Chapters 2, 5, 16]{Knapp}, and \cite[Chapter 7.4]{Folland}.
\subsection{Invariant Measures and Iwasawa Decomposition}\label{measures} The Iwasawa decomposition of $G$ is given by $G=NAK$, where $N$ and $A$ are as previously defined, and $K$ is $\mathrm{SU}(2)$, a maximal compact subgroup of $G$. Each $g\in G$ has a unique decomposition $g=n_za_tk$, with $n_z\in N$, $a_t\in A$ and $k\in K$ respectively. This decomposition gives rise to a corresponding decomposition of the Haar measure on $G$; if $g=n_za_tk$, then $d\mu_G(g)$ is a constant multiple of $e^{-2t}\,dm(z)\,dt\,dk$, where $dm$ is the Lebesgue measure on $\CC$ (i.e. $dm(x+iy)=dx\,dy$), and $dk$ is the Haar measure of $K$, normalized to be a probability measure. We choose the normalization of the Haar measure $\mu_N$ on $N$ such that if $B'\subset\CC$, and $ B=\lbrace n_z\,:\,z\in B'\rbrace$, then $\mu_N(B)=m(B')$. It will also be of use to define volumes of quotients other than $\GaG$; let $H$ be some group with Haar measure $\mu_H$, and let $\Xi$ be a discrete subgroup of $H$. Then we define $\mu_H(H/\Xi)$ (resp. $\mu_H(\Xi\backslash H)$ ) to be $\mu_H(\scrF_{\Xi})$, where $\scrF_{\Xi}$ is a fundamental domain for $H/\Xi$ (resp. $\Xi\backslash H$) in $H$.
\subsection{Lie Algebra}\label{liealg}
We denote by $\mathfrak{g}_0$ the Lie algebra of $G$. It is a 6-dimensional real semisimple Lie algebra with a basis given by
\begin{align*}
H=\frac{1}{2}\matr{1}{0}{0}{-1},  \; E_{+}=\matr{0}{1}{0}{0},\; E_{-}=\matr{0}{0}{1}{0}, \\\notag  J=\frac{1}{2}\matr{i}{0}{0}{-i},\; K_{+}=\matr{0}{i}{0}{0}, \; K_{-}=\matr{0}{0}{i}{0}.
\end{align*}
Note that
\begin{equation*}
\exp(tH)=a_t,
\end{equation*}
and
\begin{equation*}
\exp(xE_++yK_+)=n_{x+iy}.
\end{equation*}
The complexification of $\mathfrak{g}_0$ is denoted by $\mathfrak{g}$, which has the universal enveloping algebra $\scrU(\mathfrak{g})$. We use $\scrU^m(\fg)$ to denote terms in the canonical filtration of $\scrU(\mathfrak{g})$. The center of $\scrU(\mathfrak{g})$, $\scrZ(\mathfrak{g})$, contains the following two elements:
\begin{align*}
&\Omega_1 =H^2-J^2-2H+E_{+}E_{-}-K_{+}K_{-},\\\notag &\Omega_2 = 2HJ-2J+E_{+}K_{-}+K_{+}E_{-}.
\end{align*}
The following identity will play an important role for us:
\begin{equation}\label{lieequ}
H^4-4H^3+(5-\Omega_1)H^2+2(\Omega_1-1)H-(\Omega_1+\sfrac{1}{4}\Omega_2^2)=E_+U_1-K_+U_2,
\end{equation}
where $U_1$ and $U_2$ are the following elements in $\scrU^3(\mathfrak{g})$:
\begin{align}\label{lieequU}
U_1&=\sfrac 1 2 HE_-+\sfrac 1 2 JK_--\sfrac 1 4 E_+K_-^2-\sfrac 1 2 K_+E_-K_--H^2E_--HJK_-,
\\\notag U_2&=\sfrac 1 2 HK_--\sfrac 1 2 JE_-+\sfrac 1 4 K_+E_-^2-H^2K_--HJE_-.
\end{align}
\subsection{Representation Theory}\label{reptheorysec} We now recall some basic facts from the theory of unitary representations. Let $(\pi,\mathcal{H})$ be a unitary representation of $G$; i.e. $\mathcal{H}$ is a separable Hilbert space, and $\pi$ is a group homomorphism from $G$ to the group of unitary operators on $\mathcal{H}$ such that the map from $G\times\mathcal{H}$ to $\mathcal{H}$ given by
\begin{equation*}
(g,\vecv)\mapsto \pi(g)\vecv
\end{equation*}
is continuous. The representation $(\pi,\mathcal{H})$ is said to be \emph{irreducible} if $\mathcal{H}$ has no non-trivial proper closed subspace $V$ such that $\pi(G)V\subset V$. Each unitary representation  $(\pi,\mathcal{H})$ of $G$ has a direct integral decomposition
\begin{equation}\label{integraldecomp1}
(\pi,\mathcal{H})\cong \left(\int_{\mathsf{Z}}^{\oplus} \mathcal{\pi}_{\zeta}\,d\upsilon(\zeta), \int_{\mathsf{Z}}^{\oplus} \mathcal{H}_{\zeta}\,d\upsilon(\zeta) \right),
\end{equation}
where $\mathsf{Z}$ is a locally compact Hausdorff space, $\upsilon$ is a positive Radon measure on $\mathsf{Z}$, and for $\upsilon$-a.e.\ $\zeta$, $(\pi_{\zeta}, \mathcal{H}_{\zeta})$ is an irreducible unitary representation of $G$ (cf.\ eg.\ \cite[Theorem 7.36]{Folland}). The irreducible unitary representations of $G$ are relatively easy to describe: if $(\pi,\mathcal{H})$ is an irreducible unitary representation of $G$, then $(\pi,\mathcal{H})$ is isomorphic to either the trivial representation $(\pi_{\mathit{triv}},\CC)$, a \emph{principal series representation} $\mathcal{P}^{(n,\nu)}$, where $(n,\nu)\in \lbrace 0\rbrace\times i\RR_{\geq 0}\cup\NN_{>0}\times i\RR$, or a \emph{complementary series representation} $\scrP^{(0,\nu)}$, where $\nu\in (0,2)$ (see \cite[Theorem 16.2]{Knapp}). The \emph{spherical} unitary dual (the representations with a $K$-invariant vector) consists of the trivial representation and the representations $\scrP^{(0,\nu)}$, where $\nu\in i\RR_{\geq 0}\cup (0,2)$, and the \emph{tempered} unitary dual consists of the principal series representations.

We let $\scrH^{\infty}$ denote the space of \emph{smooth vectors} for $(\pi,\scrH)$; these are the vectors $\vecv$ for which the map from $G$ to $\scrH$ given by  $g\mapsto \pi(g)\vecv$ is a $C^{\infty}$ function on $G$. It is well-known that $\scrH^{\infty}$ is dense in $\scrH$. For $X\in \mathfrak{g}_0$, we define the operator $d\pi(X)$ on $\scrH^{\infty}$ as 
\begin{equation}\label{diff1}
d\pi(X)\vecv:=\left.\frac{d}{dt}\right|_{t=0}\!\!\!\! \pi\big(\exp(tX)\big)\vecv,\qquad \vecv\in\scrH^{\infty}.
\end{equation} 
We can extend this as a Lie algebra representation of $\mathfrak{g}$ in the obvious way (i.e. for $X\in\mathfrak{g}_0$, let $d\pi(iX)\vecv:=id\pi(X)\vecv$) and then to $\scrU^m(\mathfrak{g})$ by composition. We can now define Sobolev norms for the representation $(\pi,\mathcal{H})$ in the following manner: fix a basis $X_1,\dots X_6$ for $\mathfrak{g}_0$. For $\vecv\in \scrH^{\infty}$, define
\begin{equation}\label{sobnorm}
\|\vecv\|_{W^m(\mathcal{H})}^2:=\sum_{U} \|d\pi(U)\vecv\|^2_{\mathcal{H}},
\end{equation}
 where the sum runs over all $U$ that are monomials in $X_1,\dots X_6$ of degree less than or equal to $m$, including the term ``1"of order zero (i.e. $\|\vecv\|_{\scrH}^2$ is one of the summands in the right-hand side of \eqref{sobnorm}). It is easy to check that for any $m\geq 0$, there is a continuous function $C: G\rightarrow \RR_{>0}$ (independent of $(\pi,\scrH)$) such that
 \begin{equation}\label{transgbd}
 \|\pi(g)\vecv\|_{W^m(\scrH)}\leq C(g)\|\vecv\|_{W^m(\scrH)}\qquad\forall g\in G,\,\vecv\in \scrH^{\infty}.
 \end{equation}
  
We note that given the direct integral decomposition of $(\pi,\mathcal{H})$ into irreducibles \eqref{integraldecomp1}, for $\vecv\in\mathcal{H}$ we have
 \begin{equation*}
 \|\vecv\|_{\mathcal{H}}^2=\int_{\mathsf{Z}} \|\vecv_{\zeta}\|_{\mathcal{H}_{\zeta}}^2\,d\upsilon (\zeta),
 \end{equation*}
and
\begin{equation*}
\pi(g)\vecv=\int_{\mathsf{Z}} \pi_{\zeta}(g)\vecv_{\zeta}\,d\upsilon (\zeta).
\end{equation*}
Also, for $\vecv\in\scrH^{\infty}$, 
\begin{equation*}
 \|\vecv\|_{W^m(\mathcal{H})}^2=\int_{\mathsf{Z}} \|\vecv_{\zeta}\|_{W^m(\mathcal{H}_{\zeta})}^2\,d\upsilon (\zeta).
\end{equation*}
The direct integral decomposition of $(\pi,\mathcal{H})$ allows the construction of intertwining operators in the following manner: let $f$ be a bounded, continuous function from $\mathsf{Z}$ into $\CC$. We can then form the following operator: for $\vecv\in\mathcal{H}$, define
\begin{equation}\label{inter1}
T_f\vecv:=\int_{\mathsf{Z}}f(\zeta)\vecv_{\zeta}\,d\upsilon (\zeta).
\end{equation}
Then for all $g\in G$, $\vecv\in\mathcal{H}$;
\begin{equation*}
T_f\pi(g)\vecv=\int_{\mathsf{Z}}\pi_{\zeta}(g)f(\zeta)\vecv_{\zeta}\,d\upsilon (\zeta)=\pi(g)T_f\vecv.
\end{equation*}
We will also need intertwining operators of this kind where the scalar function is not necessarily uniformly bounded. By dropping the requirement that the function $f$ in \eqref{inter1} is uniformly bounded, we get operators that need not be defined on all of $\scrH$, but may be bounded operators on $\scrH^{\infty}$ with respect to various Sobolev norms $\|\cdot\|_{W^m(\scrH)}$.
 
Finally, we recall that if $(\pi,\mathcal{H})$ is irreducible, then by Schur's lemma, elements of $\mathcal{Z}(\mathfrak{g})$ act as scalars on $\mathcal{H}^{\infty}$.  If $(\pi,\mathcal{H})$ is isomorphic to $ \mathcal{P}^{(n,\nu)}$, then the scalars for $d\pi(\Omega_1)$ and $d\pi(\Omega_2)$ are
\begin{equation}\label{eigen1}
d\pi(\Omega_1)=\frac{n^2+\nu^2}{4}-1,
\end{equation}
and
\begin{equation}\label{eigen2}
d\pi(\Omega_2)=\frac{in\nu}{2}.
\end{equation}

\section{Integral Formulas}\label{intformssec}
We now prove the integral formulas for irreducible unitary representations of $G$ that will be used in Section \ref{effecsec}. In this entire section we let $(\pi,\scrH)$ be an irreducible unitary representation of $G$. We also fix a compact subset $B'$ of $\CC$, such that $m(B')>0$ and the boundary $\partial B'$ is a piecewise smooth simple closed curve. As before, we set
\begin{equation*}
B:=\lbrace n_z\;:\; z\in B' \rbrace\subset N.
\end{equation*}
For each vector $\vecv\in\scrH$ we define a function $\psi_{\vecv}:G\rightarrow \scrH$ by
\begin{equation}\label{Fdef}
\psi_{\vecv}(g):=\frac{1}{\mu_N(B)}\int_B \pi(ng)\vecv\,d\mu_N(n).
\end{equation}
We note that we have
\begin{equation*}
\psi_{\vecv}(g)=\frac{1}{m(B')}\int_{B'} \pi(n_zg)\vecv\,dm(z)=\frac{1}{m(B')}\iint\limits_{B'} \pi(n_{x+iy}g)\vecv\,dx\,dy
\end{equation*}
(for the second equality we identify $\CC$ with $\RR^2$). Note that $\psi_{\vecv}(g)$ depends linearly on $\vecv$. We apply this to \eqref{lieequ}, giving that for $\vecv\in \scrH^{\infty}$,
\begin{align}\label{diff2}
\psi_{d\pi(H^4)\vecv}(g)-4\psi_{d\pi(H^3)\vecv}(g)+(5-\lambda_1)\psi_{d\pi(H^2)\vecv}(g)+2(\lambda_1-&1)\psi_{d\pi(H)\vecv}(g)-(\lambda_1+\sfrac{1}{4}\lambda_2^2)\psi_{\vecv}(g)\\\notag&=\psi_{d\pi(E_+U_1)\vecv}(g)-\psi_{d\pi(K_+U_2)\vecv}(g),
\end{align}
where $\lambda_1=d\pi(\Omega_1)$ and $\lambda_2=d\pi(\Omega_2)$ are the scalars given at the end of Section \ref{reptheorysec} corresponding to $(\pi,\scrH)$. We now restrict ourselves to studying the behaviour of $\psi_{\vecv}$ on $A$; we define the following function from $\RR$ to $\scrH$:
\begin{equation*}
f_{\vecv}(t):=\psi_{\vecv}(a_t).
\end{equation*}
We compute the various terms of \eqref{diff2} for $f_{\vecv}$:
\begin{align*}
f_{d\pi(H)\vecv}(t)=&\frac{1}{\mu_N(B)}\int_B \pi(na_t)d\pi(H)\vecv\,d\mu_N(n)=\frac{1}{\mu_N(B)}\int_B \pi(na_t)\left.\frac{d}{dr}\right|_{r=0} \!\!\!\!\pi(a_r)\vecv\,d\mu_N(n)\\\notag &\!\!= \left.\frac{d}{dr}\right|_{r=0}\frac{1}{\mu_N(B)}\int_B \pi(na_{t+r})\vecv\,d\mu_N(n)= \frac{d}{dt}\frac{1}{\mu_N(B)}\int_B \pi(na_{t})\vecv\,d\mu_N(n)=f_{\vecv}'(t).
\end{align*}
Hence
\begin{equation}\label{hdiff}
f_{d\pi(H^m)\vecv}(t)=f_{\vecv}^{(m)}(t).
\end{equation}
We also have
\begin{align}\label{xdiff}
f_{d\pi(E_+)\vecv}(t)=&\frac{1}{m(B')}\iint\limits_{B'} \!\pi(n_{x+iy}a_t)d\pi(E_+)\vecv\,dx\,dy=\frac{1}{m(B')}\iint\limits_{B'} \pi(n_{x+iy}a_t)\left.\frac{d}{dr}\right|_{r=0} \!\!\!\!\pi(n_r)\vecv\,dx\,dy\\\notag&=\frac{1}{m(B')}\iint\limits_{B'}\left.\frac{d}{dr}\right|_{r=0}\!\!\!\! \pi(n_{x+re^t+iy}a_t) \vecv\,dx\,dy=\frac{e^t}{m(B')}\iint\limits_{B'} \frac{\partial}{\partial x}\pi(n_{x+iy}a_t) \vecv\,dx\,dy.
\end{align}
Likewise,
\begin{equation}\label{ydiff}
f_{d\pi(K_+)\vecv}(t)=\frac{e^t}{m(B')}\iint\limits_{B'} \frac{\partial}{\partial y}\pi(n_{x+iy}a_t) \vecv\,dx\,dy.
\end{equation}
Combining \eqref{diff2}, \eqref{hdiff}, \eqref{xdiff} and \eqref{ydiff} gives
\begin{align*}
f_{\vecv}^{(4)}(t)-4f_{\vecv}^{(3)}(t)+&(5-\lambda_1)f_{\vecv}^{(2)}(t)+2(\lambda_1-1)f_{\vecv}^{(1)}(t)-(\lambda_1+\sfrac{1}{4}\lambda_2^2)f_{\vecv}(t)\\&=\frac{e^t}{m(B')}\iint\limits_{B'}\left( \frac{\partial}{\partial x}\pi(n_{x+iy}a_t)d\pi(U_1)\vecv-\frac{\partial}{\partial y}\pi(n_{x+iy}a_t)d\pi(U_2)\vecv\right)dx\,dy.
\end{align*}
By Green's Theorem, we then have
\begin{align}\label{ODE1}
f_{\vecv}^{(4)}(t)-4f_{\vecv}^{(3)}(t)+(5&-\lambda_1)f_{\vecv}^{(2)}(t)+2(\lambda_1-1)f_{\vecv}^{(1)}(t)-(\lambda_1+\sfrac{1}{4}\lambda_2^2)f_{\vecv}(t)\\\notag&=\frac{e^t}{m(B')}\oint\limits_{\partial B'} \pi(n_{x+iy}a_t)d\pi(U_2)\vecv\,dx+\pi(n_{x+iy}a_t)d\pi(U_1)\vecv\,dy.
\end{align}
In Proposition \ref{Intformlem}, we present an integral representation of the solution to this (Hilbert space-valued) ODE that will prove to be useful in obtaining asymptotics for $f_{\vecv}(t)$ as $t$ tends towards $-\infty$. We first note, however, that if $(\pi,\mathcal{H})$ is one of the irreducible representations listed in Section \ref{reptheorysec} such that $\lambda_2=0$, we do not need to solve a fourth order differential equation. Indeed, in this case we may use the following identity
\begin{equation}\label{lieequ2}
H^3-3H^2+(2-\Omega_1)H+\Omega_1=\sfrac{1}{2}\Omega_2J+E_+V_1+K_+V_2,
\end{equation}
where $V_1$, $V_2$ are the following elements of $\scrU^2(\fg)$:
\begin{equation}\label{lieequ2V}
V_1=E_--E_-H-\sfrac{1}{2}K_-J,\qquad V_2=-K_-+K_-H-\sfrac{1}{2}E_-J. 
\end{equation}
In the same way that \eqref{lieequ} implies \eqref{ODE1}, \eqref{lieequ2} implies, when $d\pi(\Omega_2)=\lambda_2=0$:
\begin{align}\label{ODE2}
f_{\vecv}^{(3)}(t)-3f_{\vecv}^{(2)}(t)+(2-\lambda_1)&f_{\vecv}^{(1)}(t)+\lambda_1f_{\vecv}(t)\\\notag&=\frac{e^{t}}{m(B')}\oint_{\partial B'}\pi(n_{x+iy}a_t)d\pi(V_2)\vecv\,dx+\pi(n_{x+iy}a_t)d\pi(V_1)\vecv\,dy.
\end{align}

For notational purposes we introduce the following function: for $X,Y\in\scrU(\mathfrak{g})$, $t\in\RR$ and $\vecv\in\scrH^{\infty}$, define
\begin{equation}\label{Idef}
I_{\vecv}(X,Y,t):= \frac{e^t}{m(B')}\oint_{\partial B'}\pi(n_{x+iy}a_t)d\pi(Y)\vecv\,dx+\pi(n_{x+iy}a_t)d\pi(X)\vecv\,dy.
\end{equation}
By using the values of $\lambda_1$ and $\lambda_2$ given in \eqref{eigen1} and \eqref{eigen2}, we may rewrite \eqref{ODE1} as
\begin{equation}\label{ODE11}
\big(\sfrac{d}{dt}-(1-\sfrac{n}{2})\big)\big(\sfrac{d}{dt}-(1-\sfrac{\nu}{2})\big)(\sfrac{d}{dt}-(1+\sfrac{\nu}{2})\big)\big(\sfrac{d}{dt}-(1+\sfrac{n}{2})\big)f_{\vecv}(t)=I_{\vecv}(U_1,U_2,t).
\end{equation}

\begin{lem}\label{Iintlem}
Assume $(\pi,\scrH)\cong\scrP^{(n,\nu)}$, where $n> 0$, and $\vecv\in \scrH^{\infty}$. Let $g_{\vecv}(t)$ be defined by 
\begin{equation}\label{gdef}
 e^{(1-n/2)t}g_{\vecv}(t)=(\sfrac{d}{dt}-(1-\sfrac{\nu}{2})\big)(\sfrac{d}{dt}-(1+\sfrac{\nu}{2})\big)\big(\sfrac{d}{dt}-(1+\sfrac{n}{2})\big)f_{\vecv}(t).
\end{equation}
Then
\begin{equation}\label{grel}
g_{\vecv}(t)=\int_{-\infty}^t e^{(n/2-1)s}I_{\vecv}(U_1,U_2,s)\,ds.
\end{equation}
\end{lem}
\begin{proof}

By \eqref{ODE11} and \eqref{gdef},
\begin{equation*}
\big(\sfrac{d}{dt}-(1-\sfrac{n}{2})\big)\big(e^{(1-n/2)t}g_{\vecv}(t)\big)=I_{\vecv}(U_1,U_2,t),
\end{equation*}
so
\begin{equation*}
\sfrac{d}{dt}g_{\vecv}(t)=e^{(n/2-1)t}I_{\vecv}(U_1,U_2,t).
\end{equation*}
The fundamental theorem of calculus then gives
\begin{equation*}
g_{\vecv}(t)-g_{\vecv}(r)=\int_r^te^{(n/2-1)s}I_{\vecv}(U_1,U_2,s)\,ds.
\end{equation*}
By using the triangle inequality for integrals in \eqref{Idef}, we get that
\begin{equation}\label{Ibd}
\|I_{\vecv}(U_1,U_2,s)\|\ll_{B} e^s\|\vecv\|_{W^3(\scrH)},
\end{equation}  
so
\begin{equation*}
\|g_{\vecv}(t)-g_{\vecv}(r)\|\ll_{B,\vecv}\int_r^t e^{\frac{sn}{2}}\,ds=\frac{2}{n}\left(e^{\frac{tn}{2}}-e^{\frac{rn}{2}}\right).
\end{equation*}
From this uniform bound, we see that $g_{\vecv}(r)$ converges to some vector $\vecv_{\infty}$ as $r\rightarrow-\infty$, and
\begin{equation*}
g_{\vecv}(t)=\vecv_{\infty}+\int_{-\infty}^{t}e^{(n/2-1)s}I_{\vecv}(U_1,U_2,s)\,ds.
\end{equation*}
It remains to prove that $\vecv_{\infty}=0$. We let $\vecw=(d\pi(H)-(1-\sfrac{\nu}{2}))(d\pi(H)-(1+\sfrac{\nu}{2}))(d\pi(H)-(1+\sfrac{n}{2}))\vecv$, and note that from the definition of $g_{\vecv}(r)$, 
\begin{equation*}
 g_{\vecv}(r)=e^{(n/2-1)r}f_{\vecw}(r).
\end{equation*}
 For $n\geq 3$, have
 \begin{equation*}
 \|g_{\vecv}(r)\|=\|e^{(n/2-1)r}f_{\vecw}(r)\|\leq e^{(n/2-1)r}\|\vecw\|\leq e^{r/2}\|\vecw\|,
 \end{equation*}
 so $\vecv_{\infty}=0$. For $n=1$ or $n=2$, we use quantitative decay of matrix coefficients; let $\vecu$ be any vector in $\scrH^{\infty}$. We then have
\begin{align*}
\langle e^{(n/2-1)r}f_{\vecw}(r),\vecu\rangle=\frac{1}{\mu_N(B)}\int_{B}e^{(n/2-1)r}\langle \pi(l)\pi(a_{r})\vecw,\vecu\rangle\,d\mu_N(l)\\\notag =\frac{1}{\mu_N(B)}\int_{B}e^{(n/2-1)r}\langle \pi(a_{r})\vecw,\pi(l^{-1})\vecu\rangle\,d\mu_N(l).
\end{align*}
By \cite[Proposition 5.3]{Kont} (cf. \cite[Propositions 7.14, 7.15 (c)]{Knapp}), there exist $\eta > 1/2$ and $C_{\eta}>0$, not depending on $r$, such that 
\begin{equation*}
|\langle \pi(a_{r})\vecw,\pi(l^{-1})\vecu\rangle| \leq C_{\eta} e^{\eta r} \|\vecw\|_{W^2(\scrH)}\|\pi(l^{-1})\vecu\|_{W^2(\scrH)},
\end{equation*}
giving
\begin{equation*}
|\langle  e^{(n/2-1)r}f_{\vecw}(r),\vecu\rangle|\leq C_{\eta}e^{(\eta-1/2) r} \|\vecw\|_{W^2(\scrH)}\frac{1}{\mu_N(B)}\int_{B}\|\pi(l^{-1})\vecu\|_{W^2(\scrH)}\,d\mu_N(l).
\end{equation*}
Here the integral in the right-hand side is finite (cf.\ \eqref{transgbd}); hence $\langle  e^{(n/2-1)r}f_{\vecw}(r),\vecu\rangle\rightarrow 0$ as $r\rightarrow -\infty$, and thus $\langle \vecv_{\infty},\vecu\rangle=0$. Since $\scrH^{\infty}$ is dense in $\scrH$, $\vecv_{\infty}=0$.
\end{proof}
We are now able to prove the main result of this section:
\begin{prop}\label{Intformlem} Given an irreducible unitary representation $(\pi,\scrH)$ of $G$, there exist $\CC$-valued functions $F$, $F_0$, $F_1$, $F_2$, and elements $X_1$, $X_2$ of $\scrU^3(\fg)$, all of which depend only on the isomorphism class of $(\pi,\scrH)$, such that for any $T\geq 0$ and $\vecv\in\scrH^{\infty}$, 
\begin{equation}\label{prop4eq}
f_{\vecv}(-T)=\int_{-\infty}^{0} F(T,t)I_{\vecv}(X_1,X_2,t)\,dt+\sum_{m=0}^2 F_m(T)f_{d\pi(H^m)\vecv}(0),
\end{equation}
and the following bounds hold, with all implied constants absolute:
\begin{enumerate}[(i)]
\item If $(\pi,\scrH)\cong \scrP^{(n,\nu)}$, where $n>0$, then $|F_2(T)|\ll \frac{(1+T)e^{-T}}{n}$, for $i=0,1$, 
$|F_i(T)|\ll (1+|\nu|)(1+T)e^{-T}$, and 
\begin{equation*}
|F(T,t)|\ll\frac{1}{n^2}\begin{cases} e^{(\frac{n}{2}-1)(T+t)}&\mathrm{if}\,\, t\leq -T\\ (1+t+T)e^{-(T+t)}\quad &\mathrm{if}\,\,t\geq-T. \end{cases}
\end{equation*} 
\item If $(\pi,\scrH)\cong \scrP^{(0,\nu)}$, where $\nu\in i\RR_{\geq 0}$, then $|F_0(T)|\ll (1+|\nu|^2)(1+T^2)e^{-T}$, $|F_i(T)|\ll (1+T^2)e^{-T}$ for $i=1,2$, and
\begin{equation*}
|F(T,t)|\ll \begin{cases} 0 \quad& \mathrm{if}\,\, t\leq -T\\ (T+t)^2e^{-(T+t)}\quad&\mathrm{if}\,\,t\geq -T.\end{cases}
\end{equation*} 
\item If $(\pi,\scrH)\cong \scrP^{(0,\nu)}$, where $\nu\in (0,2)$, then $|F_i(T)|\ll \nu^{-2}e^{(\frac{\nu}{2}-1)T}$ for $i=0,1,2$, and
\begin{equation*}
|F(T,t)|\ll \frac{1}{\nu^2}\begin{cases} 0 \quad& \mathrm{if}\,\, t\leq -T\\ e^{(\frac{\nu}{2}-1)(T+t)}\quad&\mathrm{if}\,\,t\geq -T.\end{cases}
\end{equation*}
\end{enumerate}
Furthermore, in cases (ii) and (iii), $X_1$ and $X_2$ may be taken as elements of $\scrU^2(\fg)$.
\end{prop}
In the proof below we obtain completely explicit formulas for $F$, $F_0$, $F_1$, $F_2$; and we see that we can take $X_1=U_1$, $X_2=U_2$ in case (i) (cf. \eqref{lieequU}), and $X_1=V_1$, $X_2=V_2$ in cases (ii) and (iii) (cf. \eqref{lieequ2V}). In case (iii), we have allowed the bounds to blow up as $\nu\rightarrow 0$ only to allow a simple statement; in fact the stronger bounds $|F_i(T)|\ll \min\lbrace 1+T^2, \nu^{-2}\rbrace e^{(\frac{\nu}{2}-1)T}$ and $|F(T,t)|\ll \min\lbrace T+t, \nu^{-1}\rbrace^{2} e^{(\frac{\nu}{2}-1)(T+t)}$ can be deduced from the explicit formula.
\begin{proof}
Let
\begin{equation*}
\alpha_1=1-\frac{n}{2},\quad\alpha_2=1-\frac{\nu}{2},\quad\alpha_3=1+\frac{\nu}{2},\quad \alpha_4=1+\frac{n}{2}.
\end{equation*}
 We first assume that $n>0$. As in the proof of Lemma \ref{Iintlem}, we use \eqref{ODE11}, giving 
\begin{equation*}
\bigg(\prod_{i=1}^4 (\sfrac{d}{dt}-\alpha_i) \bigg)f_{\vecv}(t)=I_{\vecv}(U_1,U_2,t).
\end{equation*} 
We now define the functions $g_4(t)$, $g_3(t)$, $g_2(t)$ and $g_1(t)$ to be such that
\begin{equation*}
f_{\vecv}(t)=e^{\alpha_4t}g_4(t),
\end{equation*}
and for $3\geq i\geq 1$,
\begin{equation}\label{gdefs}
\sfrac{d}{dt} g_{i+1}(t)=e^{(\alpha_i-\alpha_{i+1})t}g_i(t).
\end{equation}
From these definitions, we see that
\begin{equation*}
\sfrac{d}{dt}g_1(t)=e^{-\alpha_1t}I_{\vecv}(U_1,U_2,t).
\end{equation*}
Iterated integration of \eqref{gdefs} gives
\begin{align*}
f_{\vecv}(-T)=&g_4(0)e^{-\alpha_4T}-g_3(0)e^{-\alpha_4T}\int_{-T}^0e^{(\alpha_3-\alpha_4)t_4}\,dt_4\\\notag&+g_2(0)e^{-\alpha_4T}\int_{-T}^0e^{(\alpha_3-\alpha_4)t_4}\int_{t_4}^0e^{(\alpha_2-\alpha_3)t_3}\,dt_3\,dt_4\\\notag&-e^{-\alpha_4T}\int_{-T}^0e^{(\alpha_3-\alpha_4)t_4}\int_{t_4}^0e^{(\alpha_2-\alpha_3)t_3}\int_{t_3}^0 e^{(\alpha_1-\alpha_2)t_2}g_1(t_2)\,dt_2\,dt_3\,dt_4.
\end{align*}
We use Lemma \ref{Iintlem} and change the order of integration to get
\begin{align}\label{intformprelim}
f_{\vecv}(-T)=&g_4(0)e^{-\alpha_4T}-g_3(0)e^{-\alpha_4T}\int_{-T}^0e^{(\alpha_3-\alpha_4)t_4}\,dt_4\\\notag&+g_2(0)e^{-\alpha_4T}\int_{-T}^0e^{(\alpha_3-\alpha_4)t_4}\int_{t_4}^0e^{(\alpha_2-\alpha_3)t_3}\,dt_3\,dt_4\\\notag&\quad+\int_{-\infty}^0\!\! I_{\vecv}(U_1,U_2,t)F(T,t)\,dt,
\end{align}
where 
\begin{equation*}
F(T,t)=-e^{-\alpha_1t-\alpha_4T}\int_{\max(t,-T)}^0\int_{-T}^{t_2}\int_{-T}^{t_3}e^{\sum_{j=1}^3(\alpha_j-\alpha_{j+1})t_{j+1}} \,dt_4\,dt_3\,dt_2.
\end{equation*}
From the definitions of the $g_i$, we have that $g_4(0)=f_{\vecv}(0)$, $g_3(0)=f_{d\pi(H)\vecv}(0)-\alpha_4f_{\vecv}(0)$, and $g_2(0)=f_{d\pi(H^2)\vecv}(0)-(\alpha_3+\alpha_4)f_{d\pi(H)\vecv}(0)+\alpha_4\alpha_3f_{\vecv}(0)$. By entering these into \eqref{intformprelim}, and collecting terms, we obtain \eqref{prop4eq}.

Turning our attention to $(\pi,\scrH)\cong \scrP^{(0,\nu)}$, from \eqref{eigen2} we see that \eqref{ODE2} holds. We then rewrite this as
\begin{equation*}
\bigg(\prod_{i=1}^3 (\sfrac{d}{dt}-\alpha_i) \bigg)f_{\vecv}(t)=I_{\vecv}(V_1,V_2,t).
\end{equation*}
We then solve this equation in the same manner as when $n>0$, the main difference is that now we integrate $I_{\vecv}(V_1,V_2,t_1)$ from $t_2$, and not from $-\infty$. This gives
\begin{equation*}
f_{\vecv}(-T)=\int_{-\infty}^{0} F(T,t)I_{\vecv}(V_1,V_2,t)\,dt+\sum_{m=0}^2 F_m(T)f_{d\pi(H^m)\vecv}(0),
\end{equation*}
where
\begin{align*}
F(T,t)&=\begin{cases}0\qquad\qquad\qquad\qquad\qquad\qquad\qquad\qquad\qquad\qquad\quad\,\,\,\,\, \mathrm{if}\,\,t<-T\\ -e^{-\alpha_1t-\alpha_3T}\int_{-T}^{t}\int_{-T}^{t_2}e^{(\alpha_1-\alpha_2)t_2+(\alpha_2-\alpha_3)t_3} \,dt_3\,dt_2\qquad \mathrm{otherwise}\end{cases},\\
F_2(T)&=e^{-\alpha_3T}\int_{-T}^0e^{(\alpha_2-\alpha_3)t_3}\int_{t_3}^0e^{(\alpha_1-\alpha_2)t_2}\,dt_2\,dt_3,
\\
F_1(T)&= -e^{-\alpha_3T}\int_{-T}^0e^{(\alpha_2-\alpha_3)t_3}\,dt_3\\&\qquad-(\alpha_2+\alpha_3)e^{-\alpha_3T}\int_{-T}^0e^{(\alpha_2-\alpha_3)t_3}\int_{t_3}^0e^{(\alpha_1-\alpha_2)t_2}\,dt_2\,dt_3,
\end{align*}
and
\begin{align*}
\!\!\!\!\!\!\!\!\!\!\!\!\!\!\!\!\!\!\!\!\!\!\!F_0(T)= e^{-\alpha_3T}+&\alpha_3 e^{-\alpha_3T}\int_{-T}^0e^{(\alpha_2-\alpha_3)t_3}dt_3\\&+\alpha_2\alpha_3e^{-\alpha_3T}\int_{-T}^0e^{(\alpha_2-\alpha_3)t_3}\int_{t_3}^0e^{(\alpha_1-\alpha_2)t_2}\,dt_2\,dt_3.
\end{align*}
Entering the numerical values of the $\alpha_i$s and repeated use of the triangle inequality now give the stated bounds.  
\end{proof}
\section{The Invariant Height Function and Geometry of $\Gamma\backslash \HH^3$}\label{HgtSec} In this section we define and establish certain properties of the \emph{invariant height function}, which can be seen as measuring how far into a cusp a point in $\GaG$ is; for this reason we assume thoughout this entire section that $\Gamma$ is non-cocompact. The invariant height function will be needed for the pointwise Sobolev-type bounds of the next section. We also prove a bound on the average of the invariant height function along a translate of the boundary of $B$. This bound is stated in Proposition \ref{scrYIntlem}, and will be required when we apply Proposition \ref{Intformlem} in the proof of Theorem \ref{mainthm} (see Section \ref{effecsec}).
 
 \subsection{The Invariant Height Function}\label{invheight} We start by recalling some of the main facts (the main reference of these are \cite[Chapters 1, 2]{EGM}) regarding the action of $G$ on the hyperbolic upper half-space $\HH^3=\lbrace (z,r)\,:\,z\in\CC,\,r\in\RR_+\rbrace$. For $g=\smatr a b c d\in G$, $(z,r)\in\HH^3$, define
\begin{equation}\label{Gaction}
g\cdot (z,r):= \left(\frac{(az+b)(\bar{c}\bar{z}+\bar{d})+a\bar{c}r^2}{|cz+d|^2+|c|^2r^2},\frac{r}{|cz+d|^2+|c|^2r^2}\right). 
\end{equation} 
We also recall that this action extends uniquely to the boundary $\partial_{\infty} \HH^3=\lbrace \infty\rbrace \cup \lbrace (z,0)\,:\,z\in\CC\rbrace$. It will be convenient to view $\HH^3$ as the subset $\lbrace z+rj\,:\,z\in\CC,\,r\in\RR_+\rbrace$ of the quaternions. Letting $P=z+rj$, we can then write \eqref{Gaction} in the more concise form
\begin{equation*}
g\cdot P = \frac{aP+b}{cP+d}.
\end{equation*}
We note that  
\begin{equation*}
(n_za_tk)\cdot j=z+e^tj;
\end{equation*}
this gives the standard identification of $G/K$ with $\HH^3$. It is also useful to define
\begin{equation*}
\mathrm{ht}(z+rj):=r.
\end{equation*}
For $\eta\in \partial_{\infty} \HH^3\setminus\lbrace\infty\rbrace$ and $\delta\in\RR_{+}$, define
\begin{equation*}
\scrH(\eta,\delta):=\lbrace z+rj\;:\;|z-\eta|^2+|r-\delta/2|^2<(\delta/2)^2\rbrace;
\end{equation*}
note that this a Euclidean ball tangent to $\partial_{\infty}\HH^3$ in the upper half-space model. Define also
\begin{equation*}
\scrH(\infty,\delta):=\lbrace P\in\HH^3\;:\;\mathrm{ht}(P)>\delta\rbrace.
\end{equation*}
The sets $\scrH(\eta,\delta)$ are called \emph{horoballs}, their boundaries in $\HH^3$ are called \emph{horospheres}. Since $\Gamma$ is a non-cocompact lattice, $\Gamma\backslash\HH^3$ is a hyperbolic 3-orbifold with at least one cusp. We shall now define the \emph{invariant height function} $\scrY_{\Gamma}$ as a $\Gamma$-left and $K$-right invariant function on $G$. We may then also view $\scrY_{\Gamma}$ as a function on $\GaG$, as well as a function on $\HH^3$--we shall abuse notation and also write $\scrY_{\Gamma}(p)$ for $p\in\GaG$, and $\scrY_{\Gamma}(P)$ for $P\in\HH^3$. Recall that the cusps of $\Gamma$ (w.r.t. the action on $\HH^3$) are the parabolic fixed points of $\Gamma$ on $\partial_{\infty}\HH^3$. Let $\eta\in\partial_{\infty}\HH^3$ be a cusp of $\Gamma$, and define the following subset of $\G$:
\begin{equation*}
\NN_{\eta}^{(\Gamma)}:= \lbrace h\in G\,:\, h\cdot\eta=\infty,\,\mu_N\big(N/(N\cap h\Gamma h^{-1})\big)=1\rbrace.
\end{equation*} 
Note that if $h_1,h_2\in\NN_{\eta}^{(\Gamma)}$, then for any $g\in G$,
\begin{equation*}
\mathrm{ht}(h_1g\cdot j)=\mathrm{ht}(h_2g\cdot j).
\end{equation*}
We may therefore define
\begin{equation*}
\mathrm{ht}_{\eta}(g\cdot j):=\mathrm{ht}(hg\cdot j),\qquad  h\in \NN_{\eta}^{(\Gamma)}.
\end{equation*}
We choose a \emph{maximal} set $\eta_1,\,\eta_2,\ldots\eta_{\kappa}$ of $\Gamma$-inequivalent cusps (this is a finite set, due to $\Gamma$ being a lattice), and define, for $g\in G$:
\begin{equation}\label{scrYdef}
\scrY_{\Gamma}(g):=\max_{i=1,2\ldots\kappa}\max_{\gamma\in\Gamma}\;\mathrm{ht}_{\eta_i}\big(\gamma g\cdot j\big).
\end{equation}
Note that given $g\in G$, $p\in \GaG$, and $P\in\HH^3$ such that $p=\Gamma g$ and $P=g\cdot j$, we have $\scrY_{\Gamma}(p)=\scrY_{\Gamma}(g)=\scrY_{\Gamma}(P)$. In the proof of Lemma \ref{scrYprops} a) we will see that $\scrY_{\Gamma}$ does not depend on the choice of cusps, that is to say: given another maximal set of $\Gamma$-inequivalent cusps $\eta_1'$, $\eta_2',\ldots\eta_{\kappa}'$, and letting $\scrY_{\Gamma}'$ be defined as in \eqref{scrYdef}, but with respect to this new choice of cusps, then $\scrY_{\Gamma}=\scrY_{\Gamma}'$.    
As a function on $\HH^3$, $\scrY_{\Gamma}$ is comparable to the invariant height function defined in \cite[Section 2.3]{Sod}. We collect some properties of $\scrY_{\Gamma}$ which will be needed in the following lemma:  
\begin{lem}\label{scrYprops}$ $\\\vspace{-13pt}
\begin{enumerate}[a)]
\item $\forall g,\, g_0\in G$, $\scrY_{\Gamma}(g_0g)=\scrY_{g_0^{-1}\Gamma g_0}(g)$.
\item $\forall s\in\RR$, $\forall g\in G$: $\scrY_{\Gamma}(ga_s)\leq \max\lbrace e^s,e^{-s}\rbrace \scrY_{\Gamma}(g)$. 
\item $\forall z\in\CC$, $\forall g\in G$: $\scrY_{\Gamma}(gn_z)\leq (1+|z|^2)\scrY_{\Gamma}(g).$ 
\item Let $C_0=\sqrt{2/\sqrt{3}}$. For any two $\Gamma$-cusps $\eta\neq\eta'$ and $h\in \NN_{\eta}^{(\Gamma)}$, $h'\in\NN_{\eta'}^{(\Gamma)}$:
\begin{equation*}
 h^{-1}\cdot \scrH(\infty,C_0)\cap h'^{-1}\cdot \scrH(\infty,C_0) =\emptyset.
\end{equation*}
Consequently, for $C\geq C_0$, the set $\lbrace P\in\HH^3\,: \scrY_{\Gamma}(P)>C\rbrace$ is a disjoint union of horoballs.
\item Let $\scrC$ be a fixed compact subset of $G$. Then for all $g\in G$,
\begin{equation}\label{scrYpropd}
\sup_{h\in G}|\Gamma h \cap g \scrC|\ll_{\Gamma, \scrC} \scrY_{\Gamma}(g)^2.
\end{equation}
\end{enumerate}
\end{lem}
\begin{proof}
Starting with a), let $\eta$ be a cusp for $\Gamma$. Then for any  $g_0\in G$, $g_0^{-1}\cdot\eta$ is a cusp of $g_0^{-1}\Gamma g_0$, and $\NN_{\eta}^{(\Gamma)} g_0 = \NN_{g_0^{-1}\cdot\eta}^{(g_0^{-1}\Gamma g_0)}$. It follows from applying this with $g_0=\gamma_0\in\Gamma$ that $\max_{\gamma\in\Gamma}\mathrm{ht}_{\eta}(\gamma g\cdot j)$ is invariant under replacing $\eta$ by $\gamma_0^{-1}\cdot \eta$, for any $\gamma_0\in\Gamma$. This shows that $\scrY_{\Gamma}$ is indeed independent of the choice of representatives $\eta_1,\eta_2,\ldots\eta_{\kappa}$. Now note that $g_0^{-1}\cdot\eta_1$, $g_0^{-1}\cdot\eta_2$, $\ldots, g_0^{-1}\cdot\eta_{\kappa}$ is a maximal set of inequivalent cusps for $g_0^{-1}\Gamma g_0$, and $\max_{\gamma\in\Gamma}\,\mathrm{ht}_{\eta_i}\big(\gamma g_0g\cdot j\big)=\max_{\;\gamma\in g_0^{-1}\Gamma g_0}\mathrm{ht}_{g_0^{-1}\cdot\eta_i}\big(\gamma g\cdot j\big)$ for each $i\in\lbrace 1,2\ldots \kappa\rbrace$. Hence $\scrY_{\Gamma}(g_0g)=\scrY_{g_0^{-1}\Gamma g_0}(g)$, as claimed.\\
 
Part b) of the lemma follows from the fact that $\smatr a b c d a_s= \smatr {*} {*} {ce^{s/2}\;} {\;de^{-s/2}}$, and $(|c|^2e^s+|d|^2e^{-s})^{-1}\leq \max\lbrace e^s,e^{-s}\rbrace (|c|^2+|d|^2)^{-1}$. 
 
To prove c), it suffices to prove that $\mathrm{ht}(gn_z\cdot j)\leq \mathrm{ht}(g\cdot j)(1+|z|^2)$, i.e. that
\begin{equation*}
\frac{1}{|cz+d|^2+|c|^2}\leq \frac{1+|z|^2}{|c|^2+|d|^2},
\end{equation*}
or, equivalently, $|c|^2\leq (1+|z|^2)(|cz+d|^2+|c|^2)-|d|^2$. For given $z$ and $c$, the right-hand side of the previous inequality is minimized when $d=-cz$, giving
\begin{equation*}
(1+|z|^2)(|cz+d|^2+|c|^2)-|d|^2\geq (1+|z|^2)|c|^2-|-cz|^2=|c|^2.
\end{equation*}

For d), we may, after possibly conjugating $\Gamma$, assume that $\eta'=\infty$ and $h'=\smatr 1 0 0 1$. We then need to prove that
\begin{equation}\label{disj1}
\left\lbrace h^{-1}\cdot P\,:\,\mathrm{ht}(P)>C_0\right\rbrace\cap\left\lbrace P\,:\,\mathrm{ht}(P)>C_0\right\rbrace =\emptyset.
\end{equation}
Since $h^{-1}\cdot\infty=\eta\neq\infty$, we may write $h$ as $ \smatr a b c d$ with $c\neq 0$. Assume now that $P=z+e^t j$. Then
\begin{equation*}
\mathrm{ht}(h^{-1}\cdot P)=\mathrm{ht}\big(h^{-1}\cdot (z+e^tj)\big)=\frac{e^t}{|-cz+a|^2+|-c|^2e^{2t}}\leq \frac{1}{|c|^2e^t}.
\end{equation*} 
Since $e^t=\mathrm{ht}(P)$, we get that 
\begin{equation*}
\left\lbrace h^{-1}\cdot P\;:\; \mathrm{ht}(P)>C_0 \right\rbrace \subset \left\lbrace P\;:\; \mathrm{ht}(P)\leq \left(|c|^{2}C_0\right)^{-1} \right\rbrace.
\end{equation*}
We now see that if $\left(|c|^{2}C_0\right)^{-1}\leq C_0$, then the two sets in \eqref{disj1} will be disjoint. We will therefore prove that $|c|\geq C_0^{-1}$. Since $h\in \NN_{\eta}^{(\Gamma)}$, we have $\mu_N\big(N/(N\cap h\Gamma h^{-1})\big)=1$. By the identification of $\mu_N$ with the Lebesgue measure on $\CC$ given in Section \ref{measures}, we see that the set $\lbrace z\in\CC\;:\; h^{-1}n_zh\in\Gamma\rbrace$ is a unimodular lattice in $\CC$; there therefore exists $z_1\in\CC$, $0< |z_1|\leq C_0$, such that $h^{-1}n_{z_1}h\in\Gamma$ (cf. e.g. \cite[Chapter 1]{Conway}). The same holds for $\eta'=\infty$, $h'=\smatr 1 0 0 1$,  i.e. we can find $z_2\in\CC$,  $0< |z_2|\leq C_0$, such that $n_{z_2}\in \Gamma$. Let $\Gamma'$ be the group generated by $h^{-1}n_{z_1}h$ and $n_{z_2}$. We have $\Gamma'\subset\Gamma$, so $\Gamma'$ is a discrete subgroup of $G$, and 
\begin{equation*}
h^{-1}n_{z_1}h=\matr {\ast}{\ast}{-c^2z_1}{\ast}\in\Gamma',\quad\matr 1 {z_2} 0 1 \in\Gamma'.
\end{equation*}
Shimizu's Lemma (\cite[Theorem 3.1]{EGM}) now applies, giving $|-c^2z_1z_2|\geq 1$. Thus
\begin{equation*}
|c|\geq \frac{1}{\sqrt{ |z_1z_2|}}\geq C_0^{-1},
\end{equation*}
as desired. \\

Finally, to prove e), we start by defining $W(g):=\sup_{h\in G} |\Gamma h\cap g\scrC|$. By using the fact that for all $\gamma\in\Gamma$, $|\Gamma h \cap g\scrC|=|\Gamma (\gamma h) \cap g\scrC|$, we get
\begin{equation}\label{scrYpropdeq1}
W(g)=\sup_{h\in g\scrC} |\Gamma h\cap g\scrC|\leq |\Gamma \cap g\scrC \scrC^{-1}g^{-1}|.
\end{equation}
From this bound we see that $W$ is uniformly bounded on any compact subset of $G$. We also note that $W$, like $\scrY_{\Gamma}$, is left $\Gamma$-invariant. These two observations reduce the problem to proving that \eqref{scrYpropd} holds when $\Gamma g$ lies far out in a cusp of $\GaG$, say $\scrY_{\Gamma}(g)> C_0 e^{\Delta_{\scrC_1}}$, where $C_0$ is as in d), and
\begin{equation*}
\Delta_{\scrC_1}=\sup_{g_1\in \scrC_1}\mathrm{dist}(g_1\cdot j,j)\qquad \scrC_1:=\scrC\scrC^{-1},
\end{equation*}
$\mathrm{dist}(\cdot,\cdot)$ being the hyperbolic distance in $\HH^3$. We may now assume, by making a $\Gamma$-shift if necessary, that $\scrY_{\Gamma}(g)=\mathrm{ht}_{\eta_i}(g\cdot j)$, where $i\in\lbrace 1, 2,\ldots,\kappa\rbrace$.

For any $\gamma\in\Gamma\cap g\scrC_1 g^{-1}$ (cf. \eqref{scrYpropdeq1}), let $g_{\gamma}=g^{-1}\gamma g\in\scrC_1$. Note that $\gamma g=gg_{\gamma}$, and since $\NN_{\eta_i}^{(\Gamma)}\gamma^{-1}=\NN_{\gamma\cdot\eta_i}^{(\Gamma)}$,
\begin{align*}
&\mathrm{ht}_{\gamma\cdot\eta_i}(gg_{\gamma}\cdot j)=\mathrm{ht}_{\gamma\cdot\eta_i}(\gamma g\cdot j)=\mathrm{ht}_{\eta_i}(g\cdot j)> C_0e^{\Delta_{\scrC_1}},\\&\mathrm{ht}_{\eta_i}(gg_{\gamma}\cdot j)\geq \mathrm{ht}_{\eta_i}(g\cdot j)e^{-\mathrm{dist}(gg_{\gamma}\cdot j,g\cdot j)}\geq \mathrm{ht}_{\eta_i}(g\cdot j)e^{-\Delta_{\scrC_1}}>C_0.
\end{align*} 
From these inequalities we see that $gg_{\gamma}$ is in the intersection of the sets $\lbrace g'\in G\,:\, \mathrm{ht}_{\eta_i}(g'\cdot j)> C_0\rbrace$ and $\lbrace g'\in G\,:\, \mathrm{ht}_{\gamma\cdot\eta_i}(g'\cdot j)> C_0\rbrace$. But by d), this intersection is empty unless $\gamma\cdot \eta_i=\eta_i$. Hence $W(g)\leq |\Gamma\cap g\scrC_1 g^{-1}|=|\Gamma_{\eta_i}\cap g\scrC_1g^{-1}|$. Letting $M=\lbrace \smatr {\epsilon}{}{}{\epsilon^{-1}}\,:\,\epsilon\in\CC,\,|\epsilon|=1\rbrace$, we note that by \cite[Corollary 2.1.9]{EGM}, for $h_i\in \NN_{\eta_i}^{(\Gamma)}$, $h_i\Gamma_{\eta_i}h_i^{-1}\subset MN$. Since $\eta_i$ is a cusp of $\Gamma$, $ h_i\Gamma_{\eta_i}h_i^{-1}\cap N = \lbrace n_z\,:\, z\in\Lambda\rbrace$, where $\Lambda$ is some lattice in $\CC$. By the compactness of $M$, there exist elements $\gamma_1,\ldots\gamma_r\in  h_i\Gamma_{\eta_i}h_i^{-1}$ such that $ h_i\Gamma_{\eta_i}h_i^{-1}=\bigsqcup_{l=1}^r \gamma_l\lbrace n_z\,:\, z\in\Lambda
\rbrace$, giving
\begin{align*}
|\Gamma_{\eta_i}\cap g\scrC_1g^{-1}|=&\sum_{l=1}^r \left| h_i^{-1}\gamma_l\lbrace n_z\,:\, z\in\Lambda\rbrace h_i\cap g\scrC_1g^{-1}\right|\\&=\sum_{l=1}^r \left|\left\lbrace z\in\Lambda\,:\, \left(h_ig\right)^{-1}\gamma_ln_z h_i g \in \scrC_1\right\rbrace\right|.
\end{align*} 
Writing $h_ig\cdot j=z_g+\scrY_{\Gamma}(g)j$, and $\gamma_l^{-1}h_ig\cdot j=w_l+\scrY_{\Gamma}(g)j$, we see that for $z\in \Lambda$ such that $\left(h_ig\right)^{-1}\gamma_ln_zh_ig\in \scrC_1$,
\begin{align*}
\cosh(\Delta_{\scrC_1})\geq & \cosh\left(\mathrm{dist}\big(\left(h_ig\right)^{-1}\gamma_ln_zh_ig\cdot j,j\big)\right)\\&=\cosh\left(\mathrm{dist}\big(z+z_g+\scrY_{\Gamma}(g)j,w_l+\scrY_{\Gamma}(g)j\big)\right)\\&\quad=
1+\frac{|(w_l-z_g)-z|^2}{2\scrY_{\Gamma}(g)^2},
\end{align*} 
which implies $|(w_l-z_g)-z|\ll_{\scrC} \scrY_{\Gamma}(g)$. Since $\Lambda$ is a lattice in $\CC$, the number of such $z$ is $\ll_{\Lambda}\scrY_{\Gamma}(g)^2$, so summing over the $l$s gives $W(g)\ll \scrY_{\Gamma}(g)^2$.
\end{proof}
\subsection{Boundary Integral of the Invariant Height Function}
Using the previous lemma, we now state and prove the needed result on averages of $\scrY_{\Gamma}$ along the boundary of translates of $B$. We identify $\S^1$ with $\RR/\ZZ$, and for $x\in\S^1$, we write $|x|$ for the distance to the point $0$; in other words $|x|=\min_{m\in\ZZ}|\widetilde{x}-m|$, where $\widetilde{x}$ is any lift of $x$ to $\RR$. The main result of this section can now be stated:

\begin{prop}\label{scrYIntlem} Let $\Gamma'$ be any lattice in $G$, and let $B'$ be a connected compact subset of $\CC$ such that $0\in B'$ and there exists a piecewise smooth parametrization $\gamma:\S^1\rightarrow \CC$ of $\partial B'$ with the following properties: i) for all $t\in \S^1$ where $\gamma'(t)$ exists, $|\gamma'(t)|=L$, the arc length of $\partial B'$, and ii) there exists $c>0$ s.t. $\forall t_1,t_2\in \S^1$, $|\gamma(t_1)-\gamma(t_2)| \geq c|t_1-t_2|$. Then for $s\geq 0$,
\begin{equation*}
\int_{\partial B'} \scrY_{\Gamma'}(n_za_{-s})|dz|\ll L(1+R^2)(1+Y)+ \frac{L^2}{c}\big(1+\log \left((1+R)\left(1+Y \right)\right)+s\big) ,
\end{equation*}
where $R=\mathrm{diam}(B')$, $Y=\scrY_{\Gamma'}\smatr 1 0 0 1$, and the implied constant is absolute (in particular, the implied constant does not depend on $\Gamma'$).
\end{prop}
We split the proof of the proposition into a series of lemmas. The notation and assumptions on $B'$ from Proposition \ref{scrYIntlem} will be used throughout the remainder of this section. Recall that by Lemma \ref{scrYprops} d), the set $\lbrace P\in\HH^3\,:\,\scrY_{\Gamma'}(P)>2\rbrace$ is a disjoint union of horoballs; we let $H_{\Gamma'}$ be the family of these horoballs.
\begin{lem}\label{scrYlem1}
Let $\vartheta = 100\max\lbrace 2, Y \rbrace$. For fixed $s\geq 0$, define the following subset of $\S^1$: $\scrI:=\lbrace t\in\S^1\,:\; \scrY_{\Gamma'}(\gamma(t)+e^{-s}j)\geq \vartheta\rbrace$. Let $\scrH(\eta_1,\delta_1)$, $\scrH(\eta_2,\delta_2)$, $\ldots$, $\scrH(\eta_N,\delta_N)$ be the distinct horoballs in $H_{\Gamma'}$ that have non-empty intersection with $\gamma(\scrI)+e^{-s}j$. Then
\begin{equation*}
\int_{\partial B'} \scrY_{\Gamma'}(n_za_{-s})|dz|\ll L\vartheta + \frac{L}{c}\sum_{k=1}^N \delta_k.
\end{equation*} 
\end{lem}

\begin{proof} We partition $\scrI$ by the following sets:
\begin{equation*}
\scrI_k=\lbrace t\in \scrI\;:\;\gamma(t)+e^{-s}j\in\scrH(\eta_k,\delta_k)\rbrace\qquad 1\leq k\leq N.
\end{equation*} 
Using the assumptions on $\gamma$ and the definition of $\scrI$, we have
\begin{align}\label{srcYlemintsum1}
\int_{\partial B'} \scrY_{\Gamma'}(n_za_{-s})|dz|=&\int_{\partial B'} \scrY_{\Gamma'}(z+e^{-s}j)|dz|=\int_{\S^1} \scrY_{\Gamma'}\big(\gamma(t)+e^{-s}j\big)|\gamma'(t)|\,dt\\\notag&\leq L\vartheta + L\int_{\scrI}\scrY_{\Gamma'}\big(\gamma(t)+e^{-s}j\big)\,dt=L \vartheta + L\sum_{k=1}^N \int_{\scrI_k} \scrY_{\Gamma'}(\gamma(t)+e^{-s}j)\,dt.
\end{align} 
We now bound the contribution from the integral over one of the $\scrI_k$s. Since $\scrI_k\subset\scrI$, for $t\in\scrI_k$ we have 
\begin{equation}\label{zheight}
\scrY_{\Gamma'}(\gamma(t)+e^{-s}j)=\frac{2\delta_k e^{-s}}{|\gamma(t)-\eta_k|^2+e^{-2s}}\geq \vartheta,
\end{equation}
which implies
\begin{equation}\label{rbound}
\frac{2\delta_k}{\vartheta}\geq e^{-s}.
\end{equation}
Also, for all $t\in\scrI_k$,
\begin{equation}\label{zbound}
|\gamma(t)-\eta_k|^2 \leq\frac{2\delta_k e^{-s}}{\vartheta}-e^{-2s}=: \sigma_{k}^2.
\end{equation}
Now, let $\zeta_k(t)=\gamma(t)-\eta_k$, and define
\begin{equation*}
\scrI_{k,m}=\left\lbrace t\in \scrI_k\;:\; |\zeta_k(t)| \in \left(2^{-(m+1)}\sigma_{k},2^{-m}\sigma_{k}\right]\right\rbrace.
\end{equation*}
Then
\begin{align*}
 \int_{\scrI_k} \scrY_{\Gamma'}(\gamma(t)+e^{-s}j)\,dt=\int_{\scrI_k} \frac{2\delta_k e^{-s}}{|\zeta_k(t)|^2+e^{-2s}}\,dt&=\sum_{m=0}^{\infty}\int_{\scrI_{k,m}} \frac{2\delta_k e^{-s}}{|\zeta_k(t)|^2+e^{-2s}}\,dt\\ \leq &\sum_{m=0}^{\infty} \left(\frac{2\delta_k e^{-s}}{2^{-2(m+1)}\sigma_{k}^2+e^{-2s}}\right)\,\int_{\scrI_{k,m}}dt.
\end{align*}
Now, if $t_1$ and $t_2$ are both in $\scrI_{k,m}$, then
\begin{equation*}
|\zeta_k(t_1)-\zeta_k(t_2)|\leq 2^{1-m}\sigma_{k}.
\end{equation*}
Using $|\zeta_k(t_1)-\zeta_k(t_2)| \geq c|t_1-t_2|$ gives
\begin{equation*}
|t_1-t_2|\leq c^{-1}2^{1-m}\sigma_{k},
\end{equation*}
thus 
\begin{equation*}
\int_{\scrI_{k,m}}dt \leq c^{-1}2^{2-m}\sigma_{k}.
\end{equation*}
This now gives
\begin{align*}
 \int_{\scrI_k}\scrY_{\Gamma'}(\gamma(t)+e^{-s}j)\,dt \leq& \sum_{m=0}^{\infty} \left(\frac{2\delta_k e^{-s}}{2^{-2(m+1)}\sigma_{k}^2+e^{-2s}}\right)c^{-1}2^{2-m}\sigma_{k}\\\notag&= \frac{8\delta_k(e^{s}\sigma_{k})}{c}\sum_{m=0}^{\infty}\frac{2^{-m}}{2^{-2(m+1)}(e^s\sigma_{k})^2+1}.\end{align*}
For $e^s\sigma_k \leq 1$, this is clearly $\ll \frac{\delta_k}{c}$. If $e^s\delta_k >1$, we choose $M$ so that $2^M= e^s\sigma_k$, giving
\begin{align*}
&\leq\frac{8\delta_ke^{s}\sigma_k}{c}\left(\sum_{0\leq m\leq M}
 \frac{2^{-m}}{2^{-2(m+1) }(e^s\sigma_k)^2}+\sum_{M<m}2^{-m}\right)\\&\quad
 \ll \frac{\delta_ke^{s}\sigma_k}{c}\cdot \frac{1}{e^s\sigma_k}=\frac{\delta_k}{c}.
\end{align*} 
We have thus obtained the bound
\begin{equation*}
\int_{\scrI_k}\scrY_{\Gamma'}(\gamma(t)+e^{-s}j)\,dt \ll \frac{\delta_k}{c},
\end{equation*}
which, when entered into \eqref{srcYlemintsum1}, proves the lemma.
\end{proof}
We now wish to bound the sum $\sum_{k=1}^N \delta_k$ in Lemma \ref{scrYlem1}. In order to do this, we first prove a lemma which provides a degree of separation between points of $\partial B'+e^{-s}j$ that lie deep inside distinct horoballs. This will be used to bound the number of horoballs of a given diameter that $\gamma(\scrI)+e^{-s}j$ has non-empty intersection with.
\begin{lem}\label{diffbound}
Using the notation of Lemma \ref{scrYlem1}, suppose $z_k+e^{-s}j\in \scrH(\eta_k,\delta_k)\cap\lbrace P\in\HH^3\,:\,\scrY_{\Gamma'}(P)\geq \vartheta\rbrace$, and $z_l+e^{-s}j\in \scrH(\eta_l,\delta_l)\cap\lbrace P\in\HH^3\,:\,\scrY_{\Gamma'}(P)\geq \vartheta\rbrace$, where $k\neq l$. Then
\begin{equation*}
|z_k-z_l|\geq \left(1-\frac{4}{\vartheta}\right)\sqrt{\delta_k\delta_l}.
\end{equation*} 
\end{lem}
\begin{proof}
Since $\scrH(\eta_k,\delta_k)$ and $\scrH(\eta_l,\delta_l)$ are disjoint, by Pythagoras' theorem we have
\begin{equation*}
 |\eta_k-\eta_l|^2\geq \left( \frac{\delta_k}{2}+\frac{\delta_l}{2}\right)^2-\left( \frac{\delta_k}{2}-\frac{\delta_l}{2}\right)^2=\delta_k\delta_l. 
\end{equation*}
Hence, by the triangle inequality,
\begin{equation*}
|z_k-z_l|\geq |\eta_k-\eta_l|-|z_k-\eta_k|-|z_l-\eta_l|\geq \sqrt{\delta_k\delta_l}-\sqrt{\frac{2\delta_ke^{-s}}{\vartheta}}-\sqrt{\frac{2\delta_le^{-s}}{\vartheta}},
\end{equation*}
where we used \eqref{zbound}. By \eqref{rbound}, $e^{-s}\leq 2\vartheta^{-1}\min\lbrace \delta_k,\delta_l\rbrace$, hence 
\begin{align*}
|z_k-z_l|\geq& \sqrt{\delta_k\delta_l}-\frac{2}{\vartheta}\sqrt{\min\lbrace \delta_k,\delta_l\rbrace}\left(\sqrt{\delta_k}+\sqrt{\delta_l}\right)\\&\geq \sqrt{\delta_k\delta_l}-\frac{4}{\vartheta}\sqrt{\delta_k\delta_l}=\left(1-\frac{4}{\vartheta}\right)\sqrt{\delta_k\delta_l}.
\end{align*}
\end{proof}
 \begin{lem}\label{scrYlem2}
Retaining the notation used in the previous lemmas, let $\delta_{\mathrm{Max}}=\max_{1\leq k\leq N} \delta_k$. Then 
 \begin{equation*}
\sum_{k=1}^N \delta_k\ll \delta_{\mathrm{Max}}+ L\big(1+\log\delta_{\mathrm{Max}}+s\big).
\end{equation*} 
 \end{lem}
 Note that while $\log\delta_{\mathrm{Max}}$ may be negative, $\log\delta_{\mathrm{Max}}+s$ is always positive; in fact, from  \eqref{rbound}, $\log\delta_{k}+s>0$ for all $k$. 
 \begin{proof}
 For $m\in\NN$, we define the following sets
 \begin{equation*}
 \scrS_m:=\lbrace k\in\lbrace1,2,\dots N\rbrace\;:\; \delta_k\in (2^{-(m+1)}\delta_{\mathrm{Max}}, 2^{-m}\delta_{\mathrm{Max}}]\rbrace.
 \end{equation*}
Let $M$ be such that $2^M=\frac{2\delta_{\mathrm{Max}}}{\vartheta e^{-s}}$. For $m>M$ and $k\in\scrS_m$, we have 
\begin{equation*}
 \delta_k \leq \delta_{\mathrm{Max}}2^{-m}<\delta_{\mathrm{Max}}2^{-M}=\frac{\vartheta e^{-s}}{2}.
 \end{equation*}
 This gives
 \begin{equation*}
 \frac{2\delta_k}{\vartheta}<e^{-s},
 \end{equation*}
 contradicting \eqref{rbound}! Thus for all $m>M$, $\scrS_m=\emptyset$. 
 
 We now bound the sum of the diameters by
 \begin{equation*}
 \sum_{k=1}^N \delta_k=\sum_{0\leq m \leq M}\; \sum_{l\in\scrS_m} \delta_l \leq \!\!\sum_{0\leq m \leq M}  2^{-m}\delta_{\mathrm{Max}}|\scrS_m|.
 \end{equation*}
For each $l\in\scrS_m$ we choose an element $z_l\in\gamma(\scrI)$ such that $z_l+e^{-s}j\in\scrH(\eta_l,\delta_l)$. We may 
 assume that $\scrS_m=\lbrace l_1,l_2,\dots,l_{N_m}\rbrace$, with the numbering chosen so that the points $z_{l_1},z_{l_2}\ldots z_{l_{N_m}}$ lie in this order along the curve $\partial B'$. Assuming $N_m\geq 2$, and using Lemma \ref{diffbound} to get a lower  bound on $|z_{l_p}-z_{l_{p+1}}|$, gives
 \begin{align*}
 L\geq |z_{l_{N_m}}-z_{l_{1}}|+\sum_{p=1}^{N_m-1} |z_{l_p}-z_{l_{p+1}}|\geq \left(1-\frac{4}{\vartheta}\right) \sqrt{\delta_{l_{N_m}}\delta_{l_1}}+\sum_{p=1}^{N_m-1}\left(1-\frac{4}{\vartheta}\right)\sqrt{\delta_{l_p}\delta_{l_{p+1}}}\\\notag\geq \left(1-\frac{4}{\vartheta}\right)N_m 2^{-(m+1)}\delta_{\mathrm{Max}}.
 \end{align*}
 Thus
 \begin{align*}
|\scrS_m|= N_m \leq& \max\left\lbrace 1, 2^{m+1}\left(1-\frac{4}{\vartheta}\right)^{-1} \frac{L}{\,\delta_{\mathrm{Max}}}\right\rbrace\\&\leq1+2^{m+1}\left(1-\frac{4}{\vartheta}\right)^{-1} \frac{L}{\,\delta_{\mathrm{Max}}},
 \end{align*}
 giving
 \begin{equation*}
 \sum_{k=1}^N \delta_k\leq \sum_{0\leq m \leq M}  2^{-m}\delta_{\mathrm{Max}}|\scrS_m| \leq \delta_{\mathrm{Max}}+2\left(1-\frac{4}{\vartheta}\right)^{-1}L\left(1+M\right).
 \end{equation*}
Using $\vartheta\geq 200$ and the definition of $M$ gives
\begin{equation*}
\sum_{k=1}^N \delta_k\leq \delta_{\mathrm{Max}}+ \frac{100\,L}{49\,\log 2}\big(2\log 2+\log\delta_{\mathrm{Max}}+s\big).
\end{equation*}
\end{proof}
\begin{flushleft}
\textbf{Proof of Proposition \ref{scrYIntlem}.}
 By applying Lemmas \ref{scrYlem1} and \ref{scrYlem2}, we get \end{flushleft}
 \begin{equation}\label{prop6pfeq1}
 \int_{\partial B'} \scrY_{\Gamma'}(n_za_{-s})|dz|\ll L(\vartheta+\delta_{\mathrm{Max}})+\frac{L^2}{c}\left( 1+\log \delta_{\mathrm{Max}}+s\right),
 \end{equation}
 which reduces the problem to bounding $\delta_{\mathrm{Max}}$. Let $\eta$ be one of the $\eta_k$s associated to a horoball with diameter $\delta_{\mathrm{Max}}$. We consider two cases:
 
\textit{Case 1:} $\exists\, z\in B'$ s.t. $z+j\not\in\scrH(\eta,\delta_{\mathrm{Max}})$.\\
Let $\widetilde{z}\in \partial B'$ be such that $\widetilde{z}+e^{-s}j\in \scrH(\eta,\delta_{\mathrm{Max}})$ and $\scrY_{\Gamma'}(\widetilde{z}+e^{-s}j)\geq\vartheta$. Then $|z-\eta|^2+1\geq \delta_{\mathrm{Max}}$ and $|\widetilde{z}-\eta|^2+e^{-2s}<\frac{2\delta_{\mathrm{Max}}}{\vartheta}e^{-s}$. Also, $|z-\widetilde{z}|\leq R$. Hence \begin{align*}
 \delta_{\mathrm{Max}}&\leq 1+|z-\eta|^2 \leq 1+\left(|z-\widetilde{z}|+|\widetilde{z}-\eta|\right)^2\\&\leq 1+\Big( R+\Big(\frac{2\delta_{\mathrm{Max}}}{\vartheta}\Big)^{1/2}\Big)^2\leq 1+2R^2+\frac{4\delta_{\mathrm{Max}}}{\vartheta}\leq 1+2R^2+\frac{\delta_{\mathrm{Max}}}{50},
 \end{align*}
forcing
 \begin{equation*}
 \delta_{\mathrm{Max}}\leq\frac{50}{49} (1+2R^2).
 \end{equation*}

\textit{Case 2:} $B'+j\subset \scrH(\eta,\delta_{\mathrm{Max}})$.\\
In this case, $\smatr 1 0 0 1 \cdot j=0+j\in\scrH(\eta,\delta_{\mathrm{Max}}),$
so
\begin{equation}\label{Yj}
Y=\frac{2\delta_{\mathrm{Max}}}{|\eta|^2+1}\geq 2,
\end{equation}
and $\vartheta=100\,Y$.
We let $\widetilde{z}$ be as in Case 1, so, as before, $|\widetilde{z}-\eta|^2\leq \frac{2\delta_{\mathrm{Max}}}{\vartheta}$.
Now,
\begin{equation}\label{etabd2}
|\eta|^2\leq \left( |\widetilde{z}|+|\widetilde{z}-\eta|\right)^2\leq  \Big(R+\sqrt{\frac{2\delta_{\mathrm{Max}}}{\vartheta}}\,\Big)^2\leq 2R^2+\frac{\delta_{\mathrm{Max}}}{25\,Y}.
\end{equation}
Hence, by \eqref{Yj} and \eqref{etabd2}:
\begin{equation*}
Y\geq \frac{\delta_{\mathrm{Max}}}{1+2R^2+\frac{\delta_{\mathrm{Max}}}{25\,Y}},
\end{equation*}
giving
\begin{equation*}
\delta_{\mathrm{Max}}\leq\frac{25}{24}(1+2R^2)Y.
\end{equation*}
We then have that in both Case 1 and Case 2,
\begin{equation}\label{deltamxbd}
\delta_{\mathrm{Max}}\ll (1+2R^2)\left(1+Y \right).
\end{equation}
Entering \eqref{deltamxbd} into \eqref{prop6pfeq1} and using $\vartheta\ll 1+ Y$ gives
\begin{equation*}
 \int_{\partial B'} \scrY_{\Gamma'}(n_za_{-s})|dz|\ll L(1+R^2)(1+Y)+\frac{L^2}{c}\left( 1+\log \left((1+2R^2)\left(1+Y \right)\right)+s\right),
\end{equation*}
which proves the proposition.

\hspace*{424.6pt}\qedsymbol

Proposition \ref{scrYIntlem} will be applied for lattices $\Gamma'=g^{-1}\Gamma g$. The Sobolev inequalities of the next section will also require us to consider integrals along $\partial B'$ of $\scrY_{\Gamma}$ raised to the power of some number between zero and one. We deal with both of these issues in the following corollary: 
\begin{cor}\label{scrYcor} For all $g\in G$ and $\alpha\in [0,1]$, we have
\begin{align*}
\int_{\partial B'} \scrY_{\Gamma}(gn_za_{-s})^{\alpha}|dz|\ll_{\Gamma,\alpha}
\frac{L^2(1+R^2)}{c}\left(s+Y^{\alpha}\right),
\end{align*}
where now $Y=\scrY_{\Gamma}(g)$.
\end{cor}
\begin{proof}
We use the bound on $\gamma'$, Lemma \ref{scrYprops} a), and Jensen's inequality to get
\begin{equation*}
\int_{\partial B'} \scrY_{\Gamma}(gn_za_{-s})^{\alpha}|dz|\leq L\int_{\S^1}\scrY_{g^{-1}\Gamma g}(\gamma(t)+e^{-s}j)^{\alpha}dt\leq L\left( \int_{\S^1}\scrY_{g^{-1}\Gamma g}(\gamma(t)+e^{-s}j)dt \right)^{\alpha}.
\end{equation*}
By studying the proof of Proposition \ref{scrYIntlem}, we see that
\begin{equation*}
\left(\int_{\S^1}\scrY_{g^{-1}\Gamma g}(\gamma(t)+e^{-s}j)dt\right)^{\alpha}\!\ll\! \left(\!\!(1+R^2)(1+Y')+\frac{L}{c}\left( 1+\log \left((1+R)\left(1+Y' \right)\right)+s\right)\!\!\right)^{\alpha},
\end{equation*}
where $Y'=\scrY_{g^{-1}\Gamma g}\smatr 1 0 0 1$. Again by Lemma \ref{scrYprops} a), $Y'=Y$. We now use the fact that $1\ll_{\Gamma} Y$, and $L\geq 2|\gamma(\frac{1}{2})-\gamma(0)|\geq c$, to get 
\begin{align*}
\ll_{\Gamma,\alpha} \frac{L^2}{c}(1+R^2)Y^{\alpha}+\frac{L^2}{c}\left( 1+\log \left((1+R)\left(1+Y \right)\right)+s\right).
\end{align*}
This, after noting that $\frac{L^2}{c}\left( 1+\log \left((1+R)\left(1+Y \right)\right)\right)\ll_{\Gamma,\alpha}  \frac{L^2}{c}(1+R^2)Y^{\alpha}$, concludes the proof.
\end{proof}

\begin{remark}
Note that any bi-Lipschitz mapping of $\S^1$ to the boundary of some set $B'\subset \CC$ may be reparametrized to satisfy the conditions of Proposition \ref{scrYIntlem}; this allows us to use Corollary \ref{scrYcor}  in the proof of Theorem \ref{mainthm}.
\end{remark}
\section{Decomposition of $L^2(\GaG)$ and Sobolev inequalities}\label{decomsob}
In this section we turn our attention to $L^2(\GaG)$; the main goal is to prove pointwise bounds for functions in an appropriate Sobolev space $W^m(\GaG)$. Unlike the previous section, we allow for $\Gamma$ to be cocompact (as well as non-cocompact). If $\Gamma$ is a non-cocompact lattice in $G$, functions in $W^m(\GaG)$ will generally not be bounded; they can grow in the cusps of $\GaG$. The rate of growth will be expressed in terms of the invariant height function of Section \ref{HgtSec}. This rate will also depend on spectral properties of the given function. For this reason we start with a discussion of the decomposition of $L^2(\GaG)$ into irreducible representations. 

\subsection{\hspace*{-2pt}Decomposition of $L^2(\GaG)$}\label{L2decomp}\hspace*{-3.8pt} We now study the \emph{right-regular representation} $(\rho,L^2(\GaG))$ of $G$ on $L^2(\GaG)$. Recall that $\rho$ is the right-translation operator: for all $f\in L^2(\GaG)$, $g\in G$ and $p\in \GaG$, $\big(\rho(g)f\big)(p)=f(pg)$.

We also recall that the Lie algebra acts on $\CC$-valued functions on $\GaG$ as differential operators; for $p\in \GaG$, $X\in\mathfrak{g}_0$, and $f:\GaG\rightarrow\CC$, let
\begin{equation*}
Xf(p)=\left.\frac{d}{dt}\right|_{t=0}f\big(p\exp(tX)\big).
\end{equation*}
By allowing compositions and linear combinations over $\CC$, we get an action of $\scrU(\fg)$. Define $W^m(\GaG)$ to be the space of functions in $C^m(\GaG)$ such that $Uf\in L^2(\GaG)$ for all $U\in \scrU^m(\mathfrak{g})$. Analogously to \eqref{sobnorm}, for a fixed basis of $\fg$, we define a norm $\|\cdot\|_{W^m}$ on $W^m(\GaG)$ by $\|f\|^2=\sum_U \|Uf\|^2$, the sum being over all $U$ that are monomials in the basis elements of degree not greater than $m$. We note that for $f\in L^2(\GaG)^{\infty}$ and $U\in\scrU^m(\mathfrak{g})$, we have $Uf=d\rho(U)f$, so for such $f$ we have $\|f\|_{W^m}=\|f\|_{W^m(L^2(\GaG))}$ ($\|\cdot\|_{W^m(L^2(\GaG))}$ denotes the norm on $L^2(\GaG)^{\infty}$ defined in \eqref{sobnorm}). Moreover, $L^2(\GaG)^{\infty}$ is dense in $W^m(\GaG)$ with respect to $\|\cdot\|_{W^m}$.

Recall that we may decompose $L^2(\GaG)$ as
\begin{equation*}
L^2(\GaG)=L^2(\GaG)_{\mathit{disc}}\oplus L^2(\GaG)_{\mathit{cont}},
\end{equation*}
where $L^2(\GaG)_{\mathit{disc}}$ decomposes as a direct \emph{sum} of irreducible unitary representations, and $L^2(\GaG)_{\mathit{cont}}$ decomposes as a direct \emph{integral} of irreducible  unitary representations, with the measure associated to the integral having no atoms. We now let $L^2(\GaG)_{\mathit{cusp}}$ denote the closed $G$-invariant subspace of $L^2(\GaG)$ consisting of the \emph{cuspidal} functions. A well-known result of Gelfand and Piatetski-Shapiro (see, for example, \cite[Theorem 2]{HC}) gives that $L^2(\GaG)_{\mathit{cusp}}$ decomposes into a direct sum of irreducible unitary representations of $G$, each with finite multiplicity; $L^2(\GaG)_{\mathit{cusp}}$ is therefore contained within $L^2(\GaG)_{\mathit{disc}}$, and we write
$L^2(\GaG)_{\mathit{disc}}$ as the orthogonal sum
\begin{equation*}
L^2(\GaG)_{\mathit{disc}}=L^2(\GaG)_{\mathit{cusp}}\oplus L^2(\GaG)_{\mathit{res}},
\end{equation*}  
i.e. we let $L^2(\GaG)_{\mathit{res}}$ be the orthogonal complement of $L^2(\GaG)_{\mathit{cusp}}$ in $L^2(\GaG)_{\mathit{disc}}$. 

We will need some well-known facts regarding the connection between the spectral theory of the Laplace-Beltrami operator $\Delta$ on the hyperbolic $3$-orbifold $\scrM=\GaG/K$ and the decomposition of $L^2(\GaG)$. Let $0=\lambda_0<\lambda_1\leq\lambda_2\leq\ldots\leq\lambda_M<1$ be the eigenvalues of $-\Delta$ (acting on $L^2(\scrM)$) in $(0,1)$, counted with multiplicity (cf. eg. \cite[Chapter 6]{EGM}). These small eigenvalues may be parametrized thus: for $\lambda_m\in(0,1)$, let $\lambda_m=s_m(2-s_m)$, where $s_m\in (1,2)$. For each $s_m$, $m\in \lbrace 1,2\ldots M\rbrace$, there exists a subrepresentation $\scrC^{s_m}$ of $(\rho,L^2(\GaG))$, and $\scrC^{s_m}$ is isomorphic to $\scrP^{(0,2s_m-2)}$. Moreover, any subrepresentation of $(\rho,L^2(\GaG))$ that is isomorphic to a complementary series representation is contained in $\bigoplus_{m=1}^M\scrC^{s_m}$. This is seen by noting that a $K$-invariant vector in a subrepresentation of $(\rho,L^2(\GaG))$ that is isomorphic to a complementary series representation $\scrP^{(0,\nu)}$ may be viewed as an eigenfunction to $-\Delta$ in $L^2(\scrM)$, with eigenvalue $1-\frac{\nu^2}{4}$ (since $\Omega_1$ acts as $\Delta$ on the $K$-invariant vectors in $L^2(\GaG)^{\infty}$). Letting $\varphi_0$ be the constant function $\varphi_0\equiv 1$ on $\GaG$ gives the following decomposition of $ L^2(\GaG)_{\mathit{disc}}$:
\begin{equation*}
(\rho,L^2(\GaG)_{\mathit{disc}})=(\rho,\CC\varphi_0)\oplus\scrC^{s_1}\oplus \scrC^{s_2}\oplus\ldots \scrC^{s_M}\oplus \left(\rho,L^2(\GaG)_{\mathit{disc,temp}}\right),
\end{equation*}
where $\left(\rho,L^2(\GaG)_{\mathit{disc,temp}}\right)$ is isomorphic to a direct sum of principal series representations.  

Finally, we recall that if $\Gamma$ is cocompact, then $L^2(\GaG)_{cont}$ is zero. On the other hand, if $\Gamma$ is non-cocompact, then the irreducible unitary representations occurring in the direct integral decomposition of $L^2(\GaG)_{cont}$ are all \emph{tempered}. This may be seen by the previous identification of $K$-invariant vectors of $L^2(\GaG)$ with elements of $L^2(\scrM)$, and the fact that the continuous spectrum of $-\Delta$ in $L^2(\scrM)$ is the interval $[1,\infty)$ (cf. \cite[Chapter 6]{EGM}).
\subsection{Sobolev Estimates}  \label{Sobbds}  We now prove pointwise bounds for functions in some of the subspaces discussed in the previous section. For notational convenience, we first make the following definition:
\begin{defn}\label{compscrYdef}
For cocompact $\Gamma$, let
\begin{equation*}
\scrY_{\Gamma}(p)=1 \qquad \forall p\in \GaG.
\end{equation*}
\end{defn}

Note that by this definition, Lemma \ref{scrYprops} and Corollary \ref{scrYcor} trivially hold even for cocompact $\Gamma$. By combining Lemma \ref{scrYprops} e) and \cite[Prop. B.2]{Bern}, we get:
\begin{lem}\label{lem2} Let $p\in\GaG$ and $f\in W^4(\GaG)$. Then
\begin{equation*}
|f(p)|\ll_{\Gamma}\|f\|_{W^4}\scrY_{\Gamma}(p).
\end{equation*} 
\end{lem}
For $f\in W^5(\GaG)\cap L^2(\GaG)_{\mathit{cusp}}$, we have the following uniform bound: 
\begin{lem}\label{cusplem1} For cuspidal functions in $W^5(\GaG)$, we have
\begin{equation*}
|f(p)|\ll_{\Gamma} \|f\|_{W^5}.
\end{equation*}
\end{lem}
This follows from \cite[Lemma B.3]{Bern}; we write out the proof as a preparation for Lemma \ref{reslem1} below.
\begin{proof}
Assume that $f\in W^{5}(\GaG)\cap L^2(\GaG)_{\mathit{cusp}}$ is $\RR$-valued (for $\CC$-valued functions, we may carry out the same arguments for the real and imaginary parts). Without loss of generality, we may assume that $\Gamma$ has a cusp at infinity, normalized so that the lattice $\Lambda=\lbrace z\in\CC\,:\,n_z\in\Gamma\rbrace$ is unimodular, and $p=\Gamma g$, where $g=na_tk_0$, is in the cuspidal region around $\infty$, i.e. $\scrY_{\Gamma}(p)=e^{t}>2$ (also, $\forall n'\in N$, $\scrY_{\Gamma}(n'g)=\scrY_{\Gamma}(g)$). Since $f$ is cuspidal, and in $C^5(\GaG)$, we have
\begin{equation*}
\int_{(\Gamma\cap N)\backslash N} f(ng)\,d\mu_N(n)=0.
\end{equation*} 
Since $\Lambda$ is unimodular, there exists a fundamental parallelogram $\scrF\subset\CC$, $m(\scrF)=1$, for $\Lambda$ such that
\begin{equation*}
\int_{\scrF} f(n_zg)\,dm(z)=0.
\end{equation*}
Now, $f$ is continuous and the above integral is zero, so there exists a $z_0\in\scrF$ s.t. $f(n_{z_0}g)=0$. Using $n_{z_0}g=gk_0^{-1}n_{e^{-t}z_0}k_0=g\exp(e^{-t}X_{z_0,k_0})$, where
\begin{equation*}
X_{z,k}:=\Re(z)\,\mathrm{Ad}_{k^{-1}}(E_+)+\Im(z)\,\mathrm{Ad}_{k^{-1}}(K_+)\qquad z\in\CC,\,k\in K,
\end{equation*}
we have
\begin{align*}
f(g)=f(g)-f(n_{z_0}g)=&-\int_0^{e^{-t}}\frac{d}{dr} f\left(g\exp(rX_{z_0,k_0})\right)\,dr\\\notag& =-\int_0^{e^{-t}}(X_{z_0,k_0}f)\left(n_{e^{t}rz_0}g\right)\,dr.
\end{align*}
Hence, using Lemma \ref{lem2} and $\scrY_{\Gamma}\left(n_{e^{t}rz_0}g\right)\equiv\scrY_{\Gamma}(g)=e^t$, 
\begin{equation*}
|f(g)|\ll \sup_{z\in\scrF,k\in K}\|X_{z,k}f\|_{W^4}\ll \|f\|_{W^5},
\end{equation*}
were the last bound holds since $\scrF$ and $K$ are compact.
\end{proof}
\begin{remark}
By using higher order derivatives, one may improve the bound in Lemma \ref{cusplem1} to $\scrY_{\Gamma}(p)^{-\alpha}$ for \emph{any} $\alpha\geq 0$, provided one can use a Sobolev norm of sufficiently high order (cf. eg. \cite[Lemma 10]{HC}); however we won't need this.
\end{remark}
For non-cuspidal elements of the $\scrC^{s_m}$, we need a stronger pointwise bound than Lemma \ref{lem2} provides. By studying the $K$-type decomposition of $\scrC^{s_m}$, we are able to get the following bound:
\begin{lem}\label{reslem1}Let $s\in(1,2)$. For functions in $\scrC^{s}\cap W^5(\GaG)$, we have
\begin{equation*}
|f(p)|\ll_{\Gamma,s} \scrY_{\Gamma}(p)^{2-s}\|f\|_{W^5}.
\end{equation*}
\end{lem}
\begin{proof}
 As in \cite[Sections 4.1-4.4]{Vino} and \cite[Sections 2.1-2.3]{LeeOh}, we can decompose $\scrC^{s}$ into irreducible $K$ ($=\mathrm{SU} (2)$) representations $V_l$, each with multiplicity one in $\scrC^{s}$. For each $l$ ($l$ is a non-negative integer), $V_l$ is a $2l+1$-dimensional $K$-invariant vector space. There then exists an orthonormal basis of $\scrC^{s}$ consisting of smooth functions $\varphi_{l,j}$, $-l\leq j\leq l$, aligned with the $K$-type decomposition (i.e. $\varphi_{l,j}\in V_l$) satisfying
\begin{equation}\label{resJno}
J\varphi_{l,j}=ij\varphi_{l,j},
\end{equation}
and 
\begin{equation*}
\Omega_K\varphi_{l,j}=l(l+1)\varphi_{l,j},
\end{equation*}
where $\Omega_K=-J^2-\sfrac 1 4 ( (E_+-E_-)^2+(K_++K_-)^2)$ (cf. \cite[Section 4.4]{Vino}).
For $k\in K$, let $\mathbf{A}_l(k)$ be the \emph{unitary} $(2l+1)\times (2l+1)$ matrix defined by
\begin{equation*}
\mathbf{A}_l(k)= \big( a_{l,m,n}(k) \big)_{-l\leq m,n\leq l}\;\quad \mathrm{with}\,\, a_{l,m,n}(k)=\langle \pi(k) \varphi_{l,j}, \varphi_{l,m}\rangle.
\end{equation*}
Note that  
\begin{equation}\label{resKtrans}
\varphi_{l,j}(pk)=\sum_{m=-l}^{l} a_{l,j,m}(k)\varphi_{l,m}(p).
\end{equation} 
As in the proof of Lemma \ref{cusplem1}, we may assume $\Gamma$ has a cusp at infinity, normalized so that $\Lambda=\lbrace z\in\CC\;:n_z\in\Gamma\rbrace$ is a unimodular lattice in $\CC$, and restrict our attention to points lying in the cuspidal region at infinity (that is to say points $p=\Gamma g$, with $g=na_tk$ and $\scrY_{\Gamma}(g)=e^t$). Let $\scrF$ be a fundamental parallelogram in $\CC$ for $\Lambda$, and define
\begin{equation*}
\phi_{l,j}(g)=\int_{\scrF} \varphi_{l,j}(n_z g)\,dm(z).
\end{equation*}
By abusing notation slightly, let $\phi_{l,j}(t):=\phi_{l,j}(a_t)$. We note that $\phi_{l,j}$ is left $N$-invariant, so
\begin{equation*}
\Omega_1\phi_{l,j}(t)=\big(H^2-J^2-2H+E_{+}E_{-}-K_{+}K_{-}\big)\phi_{l,j}(t)=\big(H^2-J^2-2H\big)\phi_{l,j}(t),
\end{equation*}
and
\begin{equation*}
\Omega_2\phi_{l,j}(t)=\big(2HJ-2J+E_{+}K_{-}+K_{+}E_{-}\big)\phi_{l,j}(t)=\big(2HJ-2J\big)\phi_{l,j}(t).
\end{equation*}
In particular, since $\Omega_1\phi_{l,j}=s(s-2)\phi_{l,j}$ and $\Omega_2\phi_{l,j}=0$, we get that
\begin{equation}\label{resODE2}
\phi_{l,j}''(t)+j^2\phi_{l,j}(t)-2\phi_{l,j}'(t)=s(s-2)\phi_{l,j}(t)
\end{equation}
and
\begin{equation}\label{resODE1}
2ij\phi_{l,j}'(t)-2ij\phi_{l,j}(t)=0.
\end{equation}
Assuming $j\neq 0$, solving \eqref{resODE1} gives $\phi_{l,j}(t)=A_{l,j}e^{t}$,
for some $A_{l,j}\in\CC$. Substituting this into \eqref{resODE2} and noting that $j^2-(s-1)^2\neq 0$ (since $j^2\geq 1$, and $(s-1)^2\in(0,1)$) gives $A_{l,j}=0$ (i.e. for $j\neq 0$, $\phi_{l,j}(t)$ is identically zero). When $j=0$, we solve \eqref{resODE2} (\eqref{resODE1} gives no information), giving
\begin{equation*}
\phi_{l,0}(t)=A_{l,0}e^{(2-s)t}+B_{l,0}e^{st},
\end{equation*}
where $A_{l,0}, B_{l,0}\in\CC$. We now wish to prove that $B_{l,0}=0$. Note that
\begin{align*}
2l+1=&\sum_{j=-l}^l \|\varphi_{l,j}\|^2= \sum_{j=-l}^l \int_{\GaG}|\varphi_{l,j}(g)|^2\,d\mu_G(g)\\\notag& \gg \sum_{j=-l}^l \int_{\scrF}\int_{2}^{\infty}\int_K|\varphi_{l,j}\big(n_za_tk\big)|^2e^{-2t}\,dk\,dt\,dm(z)\\\notag &\quad \geq  \sum_{j=-l}^l \int_{2}^{\infty}e^{-2t}\int_K  \left|\int_{\scrF}\varphi_{l,j}\big(n_za_tk\big)\,dm(z) \right|^2\,dk\,dt \\ \notag & \qquad= \sum_{j=-l}^l \int_{2}^{\infty}e^{-2t}\int_K  \left| \sum_{m=-l}^l a_{l,j,m}(k)\phi_{l,m}(t) \right|^2\,dk\,dt \\\notag&\quad\qquad= \int_{2}^{\infty}|\phi_{l,0}(t) |^2\,e^{-2t}\int_K \left( \sum_{j=-l}^l  |a_{l,j,0}(k)|^2\right)\,dk\,dt \\\notag &\qquad\qquad= \int_{2}^{\infty}|\phi_{l,0}(t) |^2e^{-2t}\,dt.
\end{align*}
This forces $B_{l,0}$ to be zero, as well as $|A_{l,0}|\ll \sqrt{2l+1}$. Thus, for $g=na_tk$ in the cuspidal region at $\infty$:
\begin{equation*}
\phi_{l,j}(g)=\phi_{l,j}\big(na_tk\big)=a_{l,j,0}(k)\phi_{l,0}\big(a_t\big)=a_{l,j,0}(k)A_{l,0}e^{(2-s)t}=a_{l,j,0}(k)A_{l,0}\scrY_{\Gamma}(g)^{2-s}.
\end{equation*}
 Writing $f=\sum_{l\geq 0} \sum_{|j|\leq l} d_{l,j}\varphi_{l,j}$, and arguing as in Lemma \ref{cusplem1} (that is to say: we first consider the real and imaginary parts of $f$), let $z_0\in \scrF$ be such that 
 \begin{equation*}
\int_{\scrF}f(n_zg)\,dm(z)=f(n_{z_0}g).
 \end{equation*}
 We then have 
\begin{equation}\label{resint2}
\sum_{l\geq 0} \sum_{|j|\leq l} d_{l,j}a_{l,j,0}(k)A_{l,0}\scrY_{\Gamma}(g)^{2-s}=f(g)+\int_{0}^{e^{-2t}} X_{w_0,k}f(n_{e^{2t}rw_0}g)\,dr,
\end{equation}
with $X_{w_0,k}$ defined as in the proof of Lemma \ref{cusplem1}. As before, we can bound the integral in the right-hand side of \eqref{resint2} by $\|f\|_{W^5}$, giving
\begin{equation}\label{resbnd1}
|f(g)|\ll \|f\|_{W^5}+\scrY_{\Gamma}(g)^{2-s}\left|\sum_{l\geq 0} \sum_{|j|\leq l} d_{l,j}a_{l,j,0}(k)A_{l,0}\right|.
\end{equation}
Now, since $f$ is sufficiently smooth, we may apply $\Omega_K^2$ termwise in its $\varphi_{l,j}$ decomposition, hence 
\begin{equation*}
\|f\|_{W^{4}}^2\geq \|\Omega_K^2f\|^2=\sum_{l\geq 0} \sum_{|j|\leq l} |d_{l,j}|^2 \big(l(l+1)\big)^4.
\end{equation*} 
After recalling that $|A_{l,0}|\ll \sqrt{2l+1}$, we may use the Cauchy-Schwartz inequality to bound the sum in the right-hand side of \eqref{resbnd1} by $\|f\|_{W^4}$, giving
\begin{equation*}
|f(g)|\ll \big(1+ \scrY_{\Gamma}(g)^{2-s}\big)\|f\|_{W^5}\ll_s \scrY_{\Gamma}(g)^{2-s}\|f\|_{W^5}.
\end{equation*}
\end{proof}

\section{Proof of Theorem 1}\label{effecsec}
  
\subsection{Proof of Theorem \ref{mainthm}} 
Since $\|\cdot\|_{W^7}$ is a stronger norm than both $\|\cdot\|_{W^4}$ and $\|\cdot\|_{L^2(\GaG)}$ on $W^7(\GaG)$, Lemma \ref{lem2} and the Cauchy-Schwarz inequality imply that for fixed $p$ and $T$, $f\mapsto \frac{1}{\mu_N(B)}\int_B f(pna_{-T})\,d\mu_N(n)-\int_{\GaG} f\,d\mu$ is a bounded linear functional on $W^7(\GaG)$. By the density of $L^2(\GaG)^{\infty}$ in $W^7(\GaG)$, it is therefore sufficient to consider $f\in L^2(\GaG)^{\infty}$. From the discussion in Section \ref{L2decomp}, we decompose $(\rho,L^2(\GaG))$ as the following orthogonal sum of representations
\begin{equation*}
\left(\rho,L^2(\GaG)\right)=(\rho,\CC\varphi_0)\oplus\scrC^{s_1}\oplus \scrC^{s_2}\oplus\ldots \scrC^{s_M}\oplus \left(\rho,L^2(\GaG)_{\mathit{temp}}\right),
\end{equation*}
where $\varphi_0\equiv 1$, the $\scrC^{s_m}$, $1\leq m\leq M$, are defined as in Section \ref{L2decomp}, and $\left(\rho,L^2(\GaG)_{\mathit{temp}}\right)$ decomposes as a direct integral over the \emph{tempered} unitary dual of $G$. At the level of elements of $L^2(\GaG)$, we write this decomposition as
\begin{equation*}
f =\int_{\GaG}f\,d\mu+\sum_{m=1}^Mf_m+f_{\mathit{temp}}.
\end{equation*}
We now apply the operator $\frac{1}{\mu_N(B)}\int_{B}\rho\left(na_{-T}\right) \,d\mu_N(n)$ to the previous equation, giving
\begin{align}\label{avdecomp}
\frac{1}{\mu_N(B)}\int_{B}\rho\left(na_{-T}\right)f \,d\mu_N(n)-\int_{\GaG}f\,d\mu=\sum_{m=1}^M\frac{1}{\mu_N(B)}&\int_{B}\rho\left(na_{-T}\right)f_m \,d\mu_N(n)\\\notag&+\frac{1}{\mu_N(B)}\int_{B}\rho\left(na_{-T}\right)f_{\mathit{temp}} \,d\mu_N(n).
\end{align}
We will use Proposition \ref{Intformlem} on the various summands of \eqref{avdecomp} (recall that we have assumed that $\|f\|_{W^7}<\infty$). Since $\scrC^{s_m}\cong \scrP^{(0,2s_m-2)}$, we may apply the proposition directly to each $f_m$, giving
\begin{align}\label{fmdecomp}
\frac{1}{\mu_N(B)}\int_{B}\rho\left(na_{-T}\right)f_m \,d\mu_N(n)=\int_{-T}^0 F&(T,t)I_{f_m}(V_1,V_2,t)\,dt\\\notag&+\sum_{i=0}^2\frac{F_i(T)}{\mu_N(B)}\int_B \rho(n)(H^if_m)\,d\mu_N(n).
\end{align}
By combining Lemma \ref{reslem1} and Corollary \ref{scrYcor} (also applying Lemma \ref{scrYprops} c) to accomodate for the fact that we require $0\in B'$ in Proposition \ref{scrYIntlem}, but not in Theorem \ref{mainthm}) with the definition of $I_{f_m}$, \eqref{Idef}, we get, for $t\leq 0$,
\begin{equation*}
|I_{f_m}(V_1,V_2,t)(p)|\ll \|f_m\|_{W^7}e^t\left(|t|+\scrY_{\Gamma}(p)^{2-s_m}\right).
\end{equation*}
Similarly, by combining Lemma \ref{reslem1} with Lemma \ref{scrYprops} c), we get, for $0\leq i\leq 2$, $\forall n\in B$,
\begin{equation*}
|(H^{i}f_m)(pn)|\ll \|f_m\|_{W^7}\scrY_{\Gamma}(p)^{2-s_m}.
\end{equation*}
These two bounds, combined with the bounds in Proposition \ref{Intformlem} (iii) and Proposition \ref{reslem1}, allow the 
evaluation of both sides of \eqref{fmdecomp} at $p$, and give
\begin{align*}
\left|\frac{1}{\mu_N(B)}\int_{B}f_m\left(pna_{-T}\right) \,d\mu_N(n)\right|\ll \|f_m\|_{W^7}&(2s_m-2)^{-2}\Big\lbrace e^{(s_m-2)T}\scrY_{\Gamma}(p)^{2-s_m}\\\notag&+\int_{-T}^0 e^{(s_m-2)(T+t)}e^t\left(|t|+\scrY_{\Gamma}(p)^{2-s_m}\right)\,dt\Big\rbrace,
\end{align*}
so
\begin{equation}\label{fmbdd}
\sum_{m=1}^M\left|\frac{1}{\mu_N(B)}\int_{B}f_m\left(pna_{-T}\right) \,d\mu_N(n)\right|\ll \|f\|_{W^7} \sum_{m=1}^M e^{(s_m-2)T}\scrY_{\Gamma}(p)^{2-s_m}.
\end{equation}
 
We now wish to use the same method for $f_{\mathit{temp}}$. The fact that $f_{\mathit{temp}}$ is not necessarily contained in a single irreducible representation complicates matters, and we will need to use the intertwining operators discussed in Section \ref{reptheorysec}. Assume that we have the direct integral decomposition
\begin{equation*}
\left(\rho,L^2(\GaG)_{\mathit{temp}}\right)\cong \left(\int_{\mathsf{Z}}^{\oplus}\pi_{\zeta}\,d\upsilon(\zeta),\int_{\mathsf{Z}}^{\oplus}\scrH_{\zeta}\,d\upsilon(\zeta)\right), 
\end{equation*}
where each $(\pi_{\zeta},\scrH_{\zeta})$ is isomorphic to an element of the tempered unitary dual, and write $f_{\mathit{temp}}=\int_{\mathsf{Z}}f_{\zeta}\,d\upsilon(\zeta)$. We partition $\mathsf{Z}$ into two parts: $\mathsf{Z}_0=\lbrace \zeta\in\mathsf{Z}\,:\, (\pi_{\zeta},\scrH_{\zeta})\cong \scrP^{(0,\nu)},\,\nu\in i\RR\rbrace$, and $\mathsf{Z}_1=\lbrace \zeta\in\mathsf{Z}\,:\, (\pi_{\zeta},\scrH_{\zeta})\cong \scrP^{(n,\nu)},\,n\in \NN_ {>0}\rbrace$, and let $f_{\mathit{temp},0}=\int_{\mathsf{Z}_0}f_{\zeta}\,d\upsilon(\zeta)$, $f_{\mathit{temp},1}=\int_{\mathsf{Z}_1}f_{\zeta}\,d\upsilon(\zeta)$. We then have
\begin{align*}
\frac{1}{\mu_N(B)}&\int_{B}\rho\left(na_{-T}\right)f_{\mathit{temp}} \,d\mu_N(n)=\int_{\mathsf{Z}}\frac{1}{\mu_N(B)}\int_{B}\pi_{\zeta}\left(na_{-T}\right)f_{\zeta} \,d\mu_N(n)\,d\upsilon(\zeta)
\\\notag&\!\!\!\!\!\!=\int_{\mathsf{Z}_0}\frac{1}{\mu_N(B)}\int_{B}\pi_{\zeta}\left(na_{-T}\right)f_{\zeta} \,d\mu_N(n)\,d\upsilon(\zeta)+\int_{\mathsf{Z}_1}\frac{1}{\mu_N(B)}\int_{B}\pi_{\zeta}\left(na_{-T}\right)f_{\zeta} \,d\mu_N(n)\,d\upsilon(\zeta).
\end{align*}
We now use Proposition \ref{Intformlem}, firstly on $f_{\mathit{temp},0}$, which gives
\begin{align*}
\int_{\mathsf{Z}_0}\frac{1}{\mu_N(B)}\int_{B}\pi_{\zeta}\left(na_{-T}\right)f_{\zeta} \,d\mu_N(n)\,d\upsilon(\zeta)=&\int_{\mathsf{Z}_0}\int_{-T}^0 F^{\zeta}(T,t)I_{f_{\zeta}}(V_1,V_2,t)\,dt\,d\upsilon(\zeta)\\\notag&\!\!\!\!\!+\sum_{i=0}^2\int_{\mathsf{Z}_0}\frac{F_i^{\zeta}(T)}{\mu_N(B)}\!\!\int_B\!\! \pi_{\zeta}(n)d\pi_{\zeta}(H^i)f_{\zeta}\,d\mu_N(n)\,d\upsilon(\zeta),
\end{align*}
where we use the notation $F^{\zeta},$ $F_0^{\zeta}$, $F_1^{\zeta}$ and $F_2^{\zeta}$ to keep track of which $(\pi_{\zeta},\scrH_{\zeta})$ the scalar functions come from. These functions then define, for each $T$, $t$, intertwining operators $Q(T,t)$, $Q_0(T)$, $Q_1(T)$, and $Q_2(T)$ as in \eqref{inter1}. The bounds in Proposition \ref{Intformlem} (ii) give
\begin{equation*}
\|Q(T,t)f_{\mathit{temp},0}\|\!\ll\! e^{-(T+t)}(T+t)^2\|f_{\mathit{temp},0}\|,\quad \|Q_i(T)f_{\mathit{temp},0}\|\!\ll\! e^{-T}(1+T^2)\|f_{\mathit{temp},0}\|,\;i=1,2,
\end{equation*}
and similarly for any Sobolev norm $\|\cdot\|_{W^m}$. For $Q_0(T)$ we have, after letting $\nu_{\zeta}$ denote the $\nu$ parameter of the principal series representation isomorphic to $(\pi_{\zeta},\scrH_{\zeta})$,
\begin{align}\label{Q0bds}
\|Q_0(T)f_{\mathit{temp},0}\|=&\left(\int_{\mathsf{Z}_0}|F_0^{\zeta}(T)|^2\|f_{\zeta}\|^2_{\scrH_{\zeta}}\,d\upsilon(\zeta)\right)^{1/2}\\\notag&\ll e^{-T}(1+T^2)\left(\int_{\mathsf{Z}_0}(1+|\nu_{\zeta}|^2)^2\|f_{\zeta}\|^2_{\scrH_{\zeta}}\,d\upsilon(\zeta)\right)^{1/2}\\\notag&\ll e^{-T}(1+T^2)\left(\int_{\mathsf{Z}_0}\left(\|f_{\zeta}\|^2_{\scrH_{\zeta}}+\|d\pi_{\zeta}(\Omega_1)f_{\zeta}\|^2_{\scrH_{\zeta}}\right)\,d\upsilon(\zeta)\right)^{1/2}
\\\notag&\ll e^{-T}(1+T^2)\|f_{\mathit{temp},0}\|_{W^2}.
\end{align}
These operators intertwine with the action of $G$, giving
\begin{align*}
\frac{1}{\mu_N(B)}\int_{B}\rho\left(na_{-T}\right)f_{\mathit{temp},0} \,d\mu_N(n)=\int_{-T}^0& I_{Q(T,t)f_{\mathit{temp},0}}(V_1,V_2,t)\,dt\\&+ \sum_{i=0}^2\frac{1}{\mu_N(B)}\!\!\int_B\!\!\rho(n)(Q_i(T)H^if_{\mathit{temp},0})\,d\mu_N(n).
\end{align*}
The previously discussed bounds, combined with Lemma \ref{lem2}, 
Corollary \ref{scrYcor}, and Lemma \ref{scrYprops} c) again allow 
evaluation at $p$, and give
\begin{equation*}
|I_{Q(T,t)f_{\mathit{temp},0}}(V_1,V_2,t)(p)|\ll \|f_{\mathit{temp},0}\|_{W^6}e^{-(T+t)}(T+t)^2e^t\left(|t|+\scrY_{\Gamma}(p)\right),
\end{equation*}
and $\forall n\in B$,
\begin{equation*}
|Q_i(T)H^if_{\mathit{temp},0}(pn)|\ll \|f_{\mathit{temp},0}\|_{W^6}e^{-T}(1+T^2)\scrY_{\Gamma}(p)\qquad i=0,1,2.
\end{equation*}
Combining these bounds gives
\begin{align*}
\left|\frac{1}{\mu_N(B)}\int_{B}f_{\mathit{temp},0}\left(pna_{-T}\right) \,d\mu_N(n)\right|\ll \|f_{\mathit{temp},0}\|_{W^6}&\Big\lbrace e^{-T}(1+T^2)\scrY_{\Gamma}(p)\\\notag&\!\!\!\!\!\!+e^{-T}\int_{-T}^0 (T+t)^2\left(|t|+\scrY_{\Gamma}(p)\right)\,dt\Big\rbrace,
\end{align*}
so
\begin{align}\label{ftemp0bdd}
\left|\frac{1}{\mu_N(B)}\int_{B}f_{\mathit{temp},0}\left(pna_{-T}\right) \,d\mu_N(n)\right|\ll \|f\|_{W^6} \left\lbrace e^{-T}(1+T^3)\scrY_{\Gamma}(p)+e^{-T}T^4\right\rbrace.
\end{align}
We proceed in the same manner for $f_{\mathit{temp},1}$: Proposition \ref{Intformlem} (i) is used, and we define intertwining operators $Q$, $Q_0$, $Q_1$ and $Q_2$ w.r.t. the functions $F, F_0,F_1,F_2$ to get
\begin{align*}
\frac{1}{\mu_N(B)}\int_{B}\rho\left(na_{-T}\right)f_{\mathit{temp},1} \,d\mu_N(n)=\int_{-\infty}^0& I_{Q(T,t)f_{\mathit{temp},1}}(U_1,U_2,t)\,dt\\&+ \sum_{i=0}^2\frac{1}{\mu_N(B)}\!\!\int_B\!\!\rho(n)(Q_i(T)H^if_{\mathit{temp},1})\,d\mu_N(n).
\end{align*}
Here $\|Q_2(T)H^2f_{\mathit{temp},1}\|\ll e^{-T} (1+T)\|f_{temp,1}\|_{W^2}$, and, in a similar manner to \eqref{Q0bds}, we get $\|Q_i(T)H^{i}f_{\mathit{temp},1}\|\ll e^{-T} (1+T)\|f_{temp,1}\|_{W^{2+i}}$ for $i=1,2$. As before, we now use Lemmas \ref{lem2} and \ref{scrYprops} b) to get, for $i=0,1,2$, $\forall n\in B$
\begin{equation*}
|Q_i(T)H^{i}f_{\mathit{temp},1}(pn)|\ll \|f\|_{W^7}e^{-T}(1+T)\scrY_{\Gamma}(p).
\end{equation*}
Lemma \ref{lem2}, Corollary \ref{scrYcor}, Lemma \ref{scrYprops} c), and Proposition \ref{Intformlem} (i) give
\begin{equation*}
|I_{Q(T,t)f_{\mathit{temp},1}}(U_1,U_2,t)(p)|\ll \|f_{\mathit{temp},1}\|_{W^7}e^t\left(|t|+\scrY_{\Gamma}(p)\right)\begin{cases} e^{-\frac{1}{2}(T+t)}\!\!&\mathrm{if}\,\, t\leq -T\\ e^{-(T+t)}(1+t+T)\quad \!\!\!\!\!\!\!&\mathrm{if}\,\,t\geq-T \end{cases}.
\end{equation*}
So
\begin{align*}
\left|\frac{1}{\mu_N(B)}\int_{B}f_{\mathit{temp},1}\left(pna_{-T}\right) \,d\mu_N(n)\right|\ll &\|f_{\mathit{temp},1}\|_{W^7}\Big\lbrace e^{-T}(1+T)\scrY_{\Gamma}(p)\\\notag&+\int_{-T}^0 e^{-(T+t)}(1+t+T)e^t\left(|t|+\scrY_{\Gamma}(p)\right)\,dt\\\notag&+\int_{-\infty}^{-T} e^{-\frac{1}{2}(T+t)}e^t\left(|t|+\scrY_{\Gamma}(p)\right)\,dt\Big\rbrace,
\end{align*}
giving
\begin{align}\label{ftemp1bdd}
\left|\frac{1}{\mu_N(B)}\int_{B}f_{\mathit{temp},1}\left(pna_{-T}\right) \,d\mu_N(n)\right|\ll \|f\|_{W^7}\lbrace e^{-T}(1+T^2)\scrY_{\Gamma}(p)+e^{-T}T^3\rbrace.
\end{align}
Evaluating both sides of \eqref{avdecomp} at $p$, and using the bounds \eqref{fmbdd}, \eqref{ftemp0bdd}, and \eqref{ftemp1bdd} gives
\begin{align*}
\left| \!\frac{1}{\mu_N(B)}\!\!\int_{B}\!\! f\big(pna_{-T}\big)\,d\mu_N(n)\!-\!\!\!\int_{\GaG}\!f\,d\mu\right|\!\!\ll\! \|f\|_{W^7} \Big\lbrace\!  \big( e^{-T}\scrY_{\Gamma}(p)\big)^{2-s_1}\!\!\!+\!e^{-T}\!(1+T^3)\scrY_{\Gamma}(p)\!\!\,+\!e^{-T}T^4\! \Big\rbrace\!.
\end{align*}
\hspace{425.5pt}\qedsymbol
\subsection{The dependency on $B'$}
We now make explicit $C(\Gamma,B')$'s dependency on the set $B'$ in Theorem \ref{mainthm}, which gives our most exact result:
\begin{thmbis}{mainthm}\label{mainthmprime}
Let $B'$ be a connected compact subset of $\CC$ such that there exists a piecewise smooth parametrization $\gamma:\S^1\rightarrow \CC$ of $\partial B'$ with the following properties: i) for all $t\in \S^1$ where $\gamma'(t)$ exists, $|\gamma'(t)|=L$, the arc length of $\partial B'$, and ii) there exists $c>0$ s.t. $\forall t_1,t_2\in \S^1$, $|\gamma(t_1)-\gamma(t_2)| \geq c|t_1-t_2|$. Assume that $\mathrm{diam}(B')=R$, and let $z_0$ be a point in $B'$ with minimal distance to zero. Then the conclusion of Theorem \ref{mainthm} holds with 
\begin{equation*}
C(\Gamma,B')\ll_{\Gamma}\frac{L^2(1+R^2)(1+|z_0|^2)(1+c^{-1})}{m(B')}
.
\end{equation*}
\end{thmbis}
\begin{proof}
Going through the proof of Theorem \ref{mainthm}, we find that there are two types of bounds which implicitly depend on $B'$. The first of these are bounds of the type: for $\varphi\in W^m(\GaG)$,
\begin{equation*}
\left|\frac{1}{\mu_N(B)}\int_B \varphi(pn)\,dn\right|\ll \|\varphi\|_{W^m}\scrY_{\Gamma}(p)^{\alpha},
\end{equation*}
where $\alpha\in (0,1]$. By Lemma \ref{scrYprops} c), and either Lemma \ref{lem2} or Lemma \ref{reslem1}, we have
\begin{equation}\label{Bbd1}
\left|\frac{1}{\mu_N(B)}\int_B \varphi(pn)\,dn\right|\ll_{\Gamma} (1+\left(|z_0|+R\right)^2) \|\varphi\|_{W^m}\scrY_{\Gamma}(p)^{\alpha}.
\end{equation} 
The second type of bound where we previously neglected to explicitly write out the dependency on $B'$ are bounds of the form
\begin{equation*}
|I_{\varphi}(X,Y,t)(p)|\ll \|\varphi\|_{W^m}e^t\left(|t|+\scrY_{\Gamma}(p)^{\alpha}\right).
\end{equation*}
From \eqref{Idef}, we have
\begin{align*}
|I_{\varphi}(X,Y,t)(p)|=&\left|\frac{e^t}{m(B')}\oint_{\partial B'}\big(Y\varphi\big)(pn_{x+iy}a_t)\,dx+\big(X\varphi\big)(pn_{x+iy}a_t)\,dy\right|\\&\ll_{\Gamma} \frac{e^t\|\varphi\|_{W^m}}{m(B')}\int_{\partial B'-z_0}\scrY_{\Gamma}\big(pn_{z_0}n_za_t\big)^{\alpha}\,|dz|,
\end{align*}
where $m$ and $\alpha$ are chosen according to either Lemma \ref{lem2} or Lemma \ref{reslem1}. We now apply Corollary \ref{scrYcor}, giving
\begin{align*}
|I_{\varphi}(X,Y,t)(p)|\ll_{\Gamma,\alpha} \frac{e^t\|\varphi\|_{W^m}}{m(B')}\cdot\frac{L^2(1+R^2)}{c}\big(|t|+\scrY_{\Gamma}(gn_{z_0})^{\alpha}\big).
\end{align*}
Once again, we use Lemma \ref{scrYprops} c) to get
\begin{equation}\label{Bbd2}
|I_{\varphi}(X,Y,t)(p)|\ll_{\Gamma,\alpha}\frac{L^2(1+R^2)(1+|z_0|^2)}{m(B')\,c}\|\varphi\|_{W^m}e^t\left(|t|+\scrY_{\Gamma}(p)^{\alpha}\right).
\end{equation} 
Combining \eqref{Bbd1} and \eqref{Bbd2} gives the stated bound on $C(\Gamma,B')$.
\end{proof}

We conclude by proving a generalisation of a result stated (though not proved) on \cite[pg. 228]{Sod} on the equidistribution of translates of rectangular pieces of horospheres. Note that while \cite{Sod} requires the horosphere in question to be closed, this assumption is not needed in the following:

\begin{cor}\label{mainthmprimecor}
Let $\omega_1$, $\omega_2$ be a basis for $\CC$ over $\RR$ such that $|\omega_1|=|\omega_2|=1$. For $e^{-T}\leq\delta_1\leq\delta_2$,
\begin{align*}
\Bigg|\frac{1}{\delta_1\delta_2}\int_0^{\delta_2}\int_0^{\delta_1}f&\big(pn_{\omega_1x+\omega_2y}a_{-T}\big)\,dx\,dy-\int_{\GaG}f\,d\mu\Bigg| \\ &\ll_{\Gamma,\omega_1,\omega_2}\|f\|_{W^7}\Bigg\lbrace e^{-T}T^4+\Upsilon^{2-s_1}+(1+(T+\log\delta_1)^3)\Upsilon\Bigg\rbrace,
\end{align*}
where $\Upsilon=\scrY_{\Gamma}(p)e^{-T}(1+\delta_2^2)(1+\delta_1^{-2})$.
\end{cor}
\begin{proof} Letting $q=\lfloor \frac{\delta_2}{\delta_1}\rfloor-1\geq 0$, decomposing the integral as $\int_0^{\delta_2}\!\!\!\int_0^{\delta_1}\!\!\!=\!\!\sum_{j=0}^{q}\!\int_{j\delta_1}^{(j+1)\delta_1}\!\!\!\int_0^{\delta_1}\!+\!\int_{q\delta_1}^{\delta_2}\!\!\int_0^{\delta_1}$, and defining  $p_j=pn_{j\delta_1\omega_2}$ gives 
\begin{align*}
&\Bigg|\frac{1}{\delta_1\delta_2}\int_0^{\delta_2}\int_0^{\delta_1}f\big(pn_{\omega_1x+\omega_2y}a_{-T}\big)\,dx\,dy-\int_{\GaG}f\,d\mu\Bigg|\\&\qquad\leq \frac{\delta_1}{\delta_2}\sum_{j=0}^{q-1}\Bigg|\frac{1}{\delta_1^2}\int_0^{\delta_1}\int_0^{\delta_1}f\big(p_jn_{\omega_1x+\omega_2y}a_{-T}\big)\,dx\,dy-\int_{\GaG}f\,d\mu\Bigg|\\&\qquad+\frac{\delta_2-q\delta_1}{\delta_2}\Bigg| \frac{1}{\delta_1(\delta_2-q\delta_1)}\int_0^{\delta_2-q\delta_1}\int_0^{\delta_1}f\big(p_qn_{\omega_1x+\omega_2y}a_{-T}\big)\,dx\,dy-\int_{\GaG}f\,d\mu\Bigg|\\&\quad= \frac{\delta_1}{\delta_2}\sum_{j=0}^{q-1}\Bigg|\int_0^{1}\int_0^{1}f\big(p_jn_{\delta_1(\omega_1x+\omega_2y)}a_{-T}\big)\,dx\,dy-\int_{\GaG}f\,d\mu\Bigg|\\&\qquad+\frac{\delta_2-q\delta_1}{\delta_2}\Bigg| \frac{1}{\frac{\delta_2}{\delta_1}-q}\int_0^{\delta_2/\delta_1-q}\int_0^1 f\big(p_qn_{\delta_1(\omega_1x+\omega_2y)}a_{-T}\big)\,dx\,dy-\int_{\GaG}f\,d\mu\Bigg|.
\end{align*}
Note here that $p_jn_{\delta_1(\omega_1x+\omega_2y)}a_{-T}=p_j'n_{\omega_1x+\omega_2y}a_{-(T+\log\delta_1)}$, where $p_j':=p_ja_{\log\delta_1}$. Set $C_{\Gamma}=\max_{(u,v)\in [1,2]^2} C(\Gamma,\fB_{u,v})$, where $\fB_{u,v}$ is the parallelogram spanned by $u\omega_1$ and $v\omega_2$. By Theorem \ref{mainthmprime}, $C_{\Gamma}$ is finite, and the above expression is
\begin{align*}
&\leq C_{\Gamma}\|f\|_{W^7}\Bigg\lbrace \!\!\left(\!\frac{e^{-T}}{\delta_1}\!\right)^{2-s_1}\!\!\bigg(\!\!(1\!-\!q\delta_1\delta_2^{-1})\scrY_{\Gamma}(p_q')^{2-s_1}\!+\!\frac{\delta_1}{\delta_2}\!\sum_{j=0}^{q-1}\scrY_{\Gamma}(p_j')^{2-s_1}\!\!\bigg)+ e^{-T}T^4\\&\qquad\qquad\qquad\qquad +(1+(T+\log\delta_1)^3)\left(\frac{e^{-T}}{\delta_1}\right)\bigg((1\!-\!q\delta_1\delta_2^{-1})\scrY_{\Gamma}(p_q')+\frac{\delta_1}{\delta_2}\sum_{j=0}^{q-1}  \scrY_{\Gamma}(p_j')\bigg)\!\!\Bigg\rbrace.
\end{align*}
By Lemma \ref{scrYprops} b) and c), $\scrY_{\Gamma}(p'_j)\leq\scrY_{\Gamma}(p)(1+j^2\delta_1^2)\max\lbrace\delta_1,\delta_1^{-1}\rbrace$, so for $\alpha\in [0,1]$:
\begin{equation*}
\left(\frac{e^{-T}}{\delta_1}\right)^{\alpha}\bigg((1\!-\!q\delta_1\delta_2^{-1})\scrY_{\Gamma}(p_q')^{\alpha}+\frac{\delta_1}{\delta_2}\sum_{j=0}^{q-1}  \scrY_{\Gamma}(p_j')^{\alpha}\bigg)\leq \left(\frac{\scrY_{\Gamma}(p)e^{-T}(1+\delta_2^2)}{\delta_1\min\lbrace\delta_1,\delta_1^{-1}\rbrace}\right)^{\alpha},
\end{equation*}
which gives the desired result.
\end{proof}

\end{document}